\documentclass[12pt]{amsart}

\usepackage{amsmath}
\usepackage[nice]{nicefrac}
\usepackage{amsxtra}
\usepackage{amscd}
\usepackage{amsthm}
\usepackage{amsfonts}
\usepackage{amssymb}
\usepackage{mathrsfs}
\usepackage[all]{xy}
\usepackage[pdftex]{color,graphicx}
\usepackage{graphicx}
\usepackage{color}
\usepackage{mathrsfs}
\usepackage[english]{babel}
\usepackage{longtable}
\usepackage{array}
\usepackage{pbox}

\numberwithin{equation}{section}
\newtheorem{thm}{Theorem}[section]
\newtheorem{thmint}{Theorem}
\newtheorem{lemma}[thm]{Lemma}
\newtheorem{prop}[thm]{Proposition}
\newtheorem{corol}[thm]{Corollary}
\newtheorem{remark}[thm]{Remark}
\newtheorem*{remark*}{Remark}

\newtheorem{dfn}[thm]{Definition}

\newtheorem{conjintro}[thmint]{Conjecture}
\newtheorem{propintro}[thmint]{Proposition}

\def\Z{\mathbb Z}

\def\Q{\mathbb Q}

\def\N{\mathbb N}

\def\e{\epsilon}

\def\d{\delta}

\DeclareMathSymbol{\minus} {\mathord}{operators}{"2D}

\DeclareMathOperator\M{M} 
 
\DeclareMathOperator\Mod{mod} 
 
\DeclareMathOperator\GL{GL}

  \DeclareMathOperator\dist{dist}
\DeclareMathOperator\Di{D}
\DeclareMathOperator\Or{O}

\newcommand{\rmnum}[1]{\romannumeral #1}
\newcommand{\Rmnum}[1]{\expandafter\@slowromancap\romannumeral #1@}
\begin{document}

\title[Hyperbolic hypergeometric monodromy groups]{Hyperbolic monodromy groups for the hypergeometric equation and Cartan involutions} 
\dedicatory{To Nicholas Katz with admiration}
\author[Fuchs]{Elena Fuchs}
\author[Meiri]{Chen Meiri}
\author[Sarnak]{Peter Sarnak}
\begin{abstract}
We give a criterion which ensures that a group
generated by Cartan involutions in the automorph group
of a rational quadratic form of signature $(n-1,1)$ is ``thin'',
namely it is of infinite index in the latter. It is based on
a graph defined on the integral Cartan root vectors,
as well as Vinberg's theory of hyperbolic reflection groups.
The criterion is shown to be robust for showing that many hyperbolic
hypergeometric groups for $_nF_{n-1}$ are thin.
\end{abstract}
\maketitle

\section{Introduction}

Let $\alpha,\beta\in \mathbb Q^n$ and consider the $_nF_{n-1}$ hypergeometric differential equation
\begin{equation}\label{deq}
Du=0
\end{equation}
where
\begin{equation}\label{dop}
D=(\theta+\beta_1-1)(\theta+\beta_2-1)\cdots(\theta+\beta_n-1)-z(\theta+\alpha_1)\cdots(\theta+\alpha_n)
\end{equation}
and $\theta=z\frac{d}{dz}$.

\noindent Assuming, as we do, that $0\leq\alpha_j<1$, $0\leq \beta_j<1$, and the $\alpha$'s and $\beta$'s are distinct, the $n$-functions
\begin{equation}\label{hypfunc}
z^{1-\beta_i} {_nF_{n-1}}(1+\alpha_1-\beta_i,\dots, 1+\alpha_n-\beta_i, 1+\beta_1-\beta_i,{\stackrel{\vee}{\dots}}, 1+\beta_n-\beta_i\, | \,z)
\end{equation}
where $\vee$ denotes omit $1+\beta_i-\beta_i$, are linearly independent solutions to (\ref{deq}).  Here $_nF_{n-1}$ is the hypergeometric function
$$_nF_{n-1}(\zeta_1,\dots,\zeta_n; \eta_1,\dots,\eta_{n-1}\, |\, z)=\sum_{k=0}^\infty \frac{(\zeta_1)_k\cdots(\zeta_n)_kz^k}{(\eta_1)_k\cdots(\eta_{n-1})_k k!}$$
and $(\eta)_k=\eta(\eta+1)\cdots(\eta+k-1)$.

Equation (\ref{deq}) is regular away from $\{0,1,\infty\}$ and its monodromy group $H(\alpha,\beta)$\footnote[1]{We assume throughout that $H(\alpha,\beta)$ is primitive -- see Section~\ref{beuksum}.} is generated by the local monodromies $A,B,C$ ($C=A^{-1}B$) gotten by analytic continuation of a basis of solutions along loops about $0,\infty$, and $1$ respectively, see \cite{bh} for a detailed description.  The local monodromies of equations that come from geometry are quasi-unipotent which is one reason for our restricting $\alpha$ and $\beta$ to be rational.  We restrict further to such $H(\alpha,\beta)$'s which after a suitable conjugation are contained in $\textrm{GL}_n(\mathbb Z)$.  According to \cite{bh}, this happens if the characteristic polynomials of $A$ and $B$, whose roots are $e^{2\pi i\alpha_j}$ and $e^{2\pi i\beta_k}$ respectively, are products of cyclotomic polynomials.  In particular for each $n\geq 2$ there are only finitely many such choices for the pair $\alpha,\beta$ in $\mathbb Q^n$.  \cite{bh} also determine the Zariski closure $G=G(\alpha,\beta)$ of $H(\alpha,\beta)$ explicitly in terms of $\alpha, \beta$.  Furthermore the integrality conditions that we are imposing imply that $H(\alpha,\beta)$ is self dual so that $G(\alpha,\beta)$ is either finite, $\textrm{Sp}(n)$ ($n$ even) or $\textrm{O}(n)$.  The signature of the quadratic form in the orthogonal case is determined by the relative locations of the roots $\alpha,\beta$ (see Section~\ref{beuksum}).

Our interest is whether $H(\alpha,\beta)$ is of finite or infinite index in $G(\mathbb Z)=G(\alpha,\beta)[\mathbb Z]$.  In the first case we say that $H(\alpha,\beta)$ is \emph{arithmetic} and in the second case that it is \emph{thin}.  This distinction is important in various associated number theoretic problems (see \cite{sa1}) and this paper is concerned with understanding which case happens and which is typical.  In a given example, if $H(\alpha,\beta)$ is arithmetic one can usually verify that it is so by producing generators of a finite index subgroup of $G(\mathbb Z)$, on the other hand if $H(\alpha,\beta)$ is thin then there is no general procedure to show that it is so.  Our main result is a robust certificate for showing that certain $H(\alpha,\beta)$'s are thin.

Until recently, other than the cases where $H(\alpha,\beta)$ (or equivalently $G(\alpha,\beta)$) is finite, there were few cases for which $H(\alpha,\beta)$ itself was known.  For $n=2$ it is well known that all the $H(\alpha,\beta)$'s are arithmetic and we show that the same is true for $n=3$.  For $n=4$ Brav and Thomas \cite{BT} showed very recently that the Dwork family \cite{Dw} $\alpha=(0,0,0,0), \beta=(\frac{1}{5},\frac{2}{5},\frac{3}{5},\frac{4}{5})$ as well as six other hypergeometrics with $G=\textrm{Sp}(4)$ which correspond to families of Calabi-Yau three-folds, are thin.  In fact they show that the generators $A$ and $C$ of the above $H(\alpha,\beta)$'s play generalized ping-pong on certain subsets of $\mathbb P^3$, from which they deduce that $H(\alpha,\beta)$ is a free product and hence by standard cohomological arguments that $H(\alpha,\beta)$ is thin.  On the other hand, Venkataramana shows in \cite{ve2} that for $n$ even and 
$$\alpha=\left(\frac{1}{2}+\frac{1}{n+1},\cdots,\frac{1}{2}+\frac{n}{n+1}\right), \quad \beta=\left(0,\frac{1}{2}+\frac{1}{n},\cdots,\frac{1}{2}+\frac{n-1}{n}\right),$$
$H(\alpha,\beta)$ is arithmetic (in $\textrm{Sp}(n,\mathbb Z)$).  In particular, there are infinitely many arithmetic $H(\alpha,\beta)$'s.  In \cite{SV} many more examples with $G=\textrm{Sp}(n)$ and for which $H(\alpha,\beta)$ is arithmetic are given. Another example for which $H$ can be shown to be thin is $\alpha=(0,0,0,\frac{1}{2})$, $\beta=(\frac{1}{4},\frac{1}{4},\frac{3}{4},\frac{3}{4})$, see \cite{fmsri}.  In this case $G(\mathbb R)$ is orthogonal and has signature $(2,2)$ and $G(\mathbb Z)$ splits as a product of $\textrm{SL}_2$'s.

All of our results are concerned with the case that $G(\alpha,\beta)$ is orthogonal and is of signature $(n-1,1)$ over $\mathbb R$.  We call these hyperbolic hypergeometric monodromy groups.  There is a unique (up to a scalar multiple) integral quadratic form $f$ for which $G(\mathbb Z)=\textrm{O}_f(\mathbb Z)$, or what is the same thing an integral quadratic lattice $L$ with $\textrm{O}(L):=G(\mathbb Z)$.  In Section~\ref{quadform} we determine a commensurable quadratic sublattice explicitly which facilitates many further calculations.  In this hyperbolic setting $G(\mathbb R)=\textrm{O}_f(\mathbb R)$ acts naturally as isometries of hyperbolic $n-1$-space $\mathbb H^{n-1}$ and we will use this geometry as a critical ingredient to provide a certificate for $H(\alpha,\beta)$ being thin.  Our first result is the determination of the $\alpha,\beta$'s for which $G(\alpha,\beta)$ is hyperbolic, see Theorem~\ref{theo}.  Firstly, these only occur if $n$ is odd and for $n>9$ they are completely described by seven infinite parametric families.  For $3\leq n\leq 9$ there are sporadic examples which are listed in Tables~2 and 3 of Section~\ref{numeric}.  Our determination of the seven families is based on a reduction to \cite{bh}'s list of families of $G(\alpha,\beta)$'s which are finite (i.e. those $G(\alpha,\beta)$'s for which $G=\textrm{O}(n)$ and have signature ($n,0$)).

For $n=3$, if $H(\alpha,\beta)$ is not finite then it is hyperbolic and as we noted all $6$ of these hyperbolic groups are arithmetic.  This is verified separately for each case, there being no difficulty in deciding whether a finitely generated subgroup of $\textrm{SL}_2(\mathbb R)$ is thin or not (the latter is a double cover of $\textrm{SO}(2,1))$, see the Appendix\footnote[2]{For $n\geq 5$ there is no general algorithm known to decide this.}.  For $n\geq 5$ the hyperbolic monodromies behave differently.  Our certificate of thinness applies in these cases and it is quite robust as exemplified by 

\begin{thmint}\label{theorem1}
The two families of hyperbolic monodromies $H(\alpha,\beta)$ with $n\geq 5$ and odd
\begin{itemize}
\item[(i)] $\alpha=\left(0,\frac{1}{n+1},\frac{2}{n+1},\cdots,\frac{n-1}{2(n+1)},\frac{n+3}{2(n+1)},\cdots,\frac{n}{n+1}\right), \beta=\left(\frac{1}{2},\frac{1}{n},\frac{2}{n},\cdots,\frac{n-1}{n}\right)$
\item[(ii)] $\alpha=\left(\frac{1}{2},\frac{1}{2n-2},\frac{3}{2n-2},\cdots,\frac{2n-3}{2n-2}\right), \beta=\left(0,0,0,\frac{1}{n-2},\frac{2}{n-2},\cdots,\frac{n-3}{n-2}\right)$
\end{itemize}
are thin.
\end{thmint}

In particular infinitely many of the $H(\alpha,\beta)$'s are thin and as far as we know these give the first examples in the general monodromy group setting of thin monodromy groups for which $G$ is high dimensional and simple.

\begin{remark*}
The normalized $_nF_{n-1}$'s corresponding to (i) and (ii) above (see \cite{vill} for the normalization) are
$$\sum_{k=0}^\infty \frac{(2k!)^2(nk)!}{((n+1)k)!(k!)^3}z^k \quad\mbox{and}\quad \sum_{k=0}^\infty\frac{((2n-2)k)!(2k)!}{(k!)^3((n-1)k)!((n-2)k)!}z^k$$
respectively.  The second has integral coefficients while the first does not, hence this arithmetic feature of $F$ is not reflected in the arithmeticity of the corresponding $H$ (See the end of Section~\ref{familyclassification} for more about the integrality of the coefficients).

\end{remark*}

Our certificate for thinness applies to a number of the families and many of the sporadic examples.  The full lists that we can handle thus far are recorded in Theorems~\ref{thinfamilythm} and \ref{secondthin} of Section~\ref{mindistgraph} as well as Table~\ref{table3} of Section~\ref{numeric}.  There remain some families for which the method does not apply (at least not directly).  In any case we are led to

\begin{conjintro}\label{allthinconj}
All but finitely many hyperbolic hypergeometric $H(\alpha,\beta)$'s are thin.\footnote[3]{It is quite possible that these $H(\alpha,\beta)$'s are all thin for $n\geq 5$.}
\end{conjintro}

\vspace{0.3in}

We turn to a brief description of our methods.  In general if there is an epimorphism $\psi:G(\mathbb Z)\rightarrow K$ with $\psi(H(\alpha,\beta))$ of infinite index in $K$, then clearly $H(\alpha,\beta)$ must be of infinite index in $G(\mathbb Z)$.\footnote[4]{That useful such $\psi$'s exist in this hyperbolic setting is demonstrated below.  If $G(\alpha,\beta)$ is of real rank two or bigger, there are no useful $\psi$'s by Margulis's normal subgroup theorem \cite{mar}.}  In principle, $H^1(G(\mathbb Z))$ could be infinite which would imply that $H(\alpha,\beta)$ is thin since the latter is essentially generated by involutions, however we know very little about these cohomology groups for the full $\textrm{O}_f(\mathbb Z)$.  Instead we use variations of the quotient $\psi:\textrm{O}_f(\mathbb Z)\rightarrow \textrm{O}_f(\mathbb Z)/R_f=:K_f$ where $R_f$ is the Vinberg reflective subgroup of $\textrm{O}_f(\mathbb Z)$.  Namely, $R_f$ is the group generated by the elements of $\textrm{O}_f(\mathbb Z)$ which induce reflections of hyperbolic space $\mathbb H^{n-1}$ (that is the induced action on one of the two sheeted hyperboloids in $n$-space, see Section~\ref{cartaninvol}).  Vinberg \cite{vin} and Nikulin \cite{nik} have shown that except for rare cases $K$ is infinite, see Section~\ref{cartaninvol} for a review of their results which we use.  To apply this we need to understand the image of $H(\alpha,\beta)$ in $K$.  The key observation is that the element $C$ in $H(\alpha,\beta)$, which is a linear reflection of $n$-space, induces a Cartan involution of $\mathbb H^{n-1}$, that is it is an isometry of $\mathbb H^{n-1}$ which is an inversion in a point $p\in \mathbb H^{n-1}$.  The reflection subgroup $H_r(\alpha,\beta)$ of $H(\alpha,\beta)$ is the group generated by the Cartan involutions $B^kCB^{-k}$, $k\in\mathbb Z$ and $H/H_r$ is cyclic.  Thus essentially up to commensurability the question of whether $H(\alpha,\beta)$ is thin is a special case of deciding whether a subgroup $\Delta$ of $\textrm{O}_f(\mathbb Z)$ generated by Cartan involutions is thin or not.  We approach this by examining the image of such a $\Delta$ in $K_f$ (when the latter is infinite).  At this point the study is about $\textrm{O}_f(\mathbb Z)$, $K_f$, and a general such $\Delta$ (and $n$ needn't be odd).

We define a graph $X_f$, the ``distance graph,'' associated with root vectors of $f$ which is central to the analysis.  We  assume that $f$ is even (that is $f(x)\in 2\mathbb Z$ for $x\in L$) and for $k=2$ or $-2$ let $V_k(L)=\{v\in L\, |\, f(v)=k\}$ be the corresponding root vectors (in the case that $f(x,v)\in 2\mathbb Z$ for all $x,v\in \mathbb L$ which comes up in some cases we also allow $k$ to be $\pm 4$).  The root vectors define linear reflections lying in $\textrm{O}_f(\mathbb Z)$ given by $r_v:x\mapsto x-\frac{2f(x,v)}{f(v)}v$.  By our choice of the signature of $f$, for $v\in V_{-2}(L)$ the map $r_v$ induces a Cartan involution on $\mathbb H^{n-1}$ while for $v\in V_{2}(L)$ it induces a hyperbolic reflection on $\mathbb H^{n-1}$.  Assume that the Cartan involutions generating $\Delta$ come from root vectors in $V_{-2}(L)$.  $X_f$ has for its vertices the set $V_{-2}(L)$ and we join $v$ to $w$ if $f(v,w)=-3$, this corresponds to $v$ and $w$ having the smallest distance allowed by discreteness, as points in $\mathbb H^{n-1}$.  The graph $X_f$ is a disjoint union of its connected components $\Sigma_\alpha$ and these satisfy (see Section~\ref{graphprop} for a detailed statement)

\begin{propintro}
\begin{itemize}
\item[(i)] The components $\Sigma_{\alpha}$ consist of finitely many isomorphism types.
\item[[(ii)] Each type is either a singleton or an infinite homogeneous graph corresponding to a transitive isometric action of a Coxeter group with finite stabilizer.
\item[(iii)] If $\Sigma_{\alpha}$ is a connected component of $X_f$, then the group generated by the corresponding Cartan roots, $R_{-2}(\Sigma_{\alpha})=\langle r_v\; |\; v\in \Sigma_{\alpha}\rangle$ is up to index $2$ contained in the reflection group $R_2(L)=\{r_v\; |\; v\in V_2(L)\}$.
\end{itemize}
\end{propintro}

With this the certificate for showing that $\Delta$ is thin is clear: according to Vinberg and Nikulin, $O(L)/R_2(L)$ is infinite except in rare cases (and Nikulin has a classification of thee), hence if (iii) is satisfied for the generators of $\Delta$ then $\Delta$ must be thin.  The calculations connected with the minimal distance graph can be carried out effectively in general and even explicitly for some of our families.  The process in the general case invokes an algorithm for computing the fundamental polyhedral cell for a discrete group of motions of hyperbolic space which is generated by a finite number of reflections.  Such an algorithm is provided in Section~\ref{thinguys}\footnote[5]{It differs from Vinberg's algorithm \cite{vin}, which assumes that one has an apriori list of all the hyperbolic reflections in the group.}.

The above leads to the cases discussed in Sections~\ref{mindistgraph} and \ref{numeric} for which $H(\alpha,\beta)$ is shown to be thin.  To end we note that it is possible that a much stronger version of Vinberg's theorem in the following form is valid: For all but finituely many rational quadratic forms $f$ (at least if $n$ is large enough) the full ``Weyl subgroup'' $W(f)$ generated by \emph{all} reflections in $\textrm{O}(L)$ (that is, both those inducing hyperbolic reflections and Cartan involutions on $\mathbb H^{n-1}$) is infinite index in $\textrm{O}_f(L)$ (see Nikulin \cite{nik2}).  If this is true then Conjecture~\ref{allthinconj} would follow easily from our discussion.  In Section~\ref{cartaninvol} we give an example in dimension $4$ with $R_2(L)$ being thin and $R_{-2}(L)$ arithmetic, so there is no general commensurability between these groups.  Finally, we note that when our analysis of the thinness of $H$ succeeds, it comes with a description of $H$ as a subgroup (up to commensurability) of a geometrically finite subgroup of $R_2(L)$ and this opens the door to determine the group structure of $H$ itself.  We leave this for the future.

\section{Hyperbolic hypergeometric monodromy groups: preliminaries}
\subsection{Setup of the problem}\label{beuksum}

We begin by reviewing \cite{bh} which forms the basis of our analysis.  The setup and notation is as in the Introduction.  The starting point for studying $H(\alpha,\beta)$ is the following theorem of Levelt.
\begin{thm}[\cite{levelt}]\label{levthm} For $1\leq i\leq n$ let $\alpha_i$ and $\beta_i$ be as above.  Define the complex numbers $A_1,\dots,A_n, B_1,\dots, B_n$ to be the coefficients of the polynomials
\begin{equation*}
P(z):=\prod_{j=1}^n (z-e^{2\pi i\alpha_j})=z^n+A_1z^{n-1}+\cdots+ A_n \quad {\mbox and } \quad Q(z):=\prod_{j=1}^n (z-e^{2\pi i\beta_j})=z^n+B_1z^{n-1}+\cdots+ B_n.
\end{equation*}
Then $H(\alpha,\beta)$ is the group generated by
\begin{equation*}\label{agens}\small{
A=\left(
\begin{array}{lllll}
0&0&\cdots&0&\minus A_n\\
1&0&\cdots&0&\minus A_{n-1}\\
0&1&\cdots&0&\vdots\\
0&0&\ddots&0&\minus A_2\\
0&0&\cdots&1&\minus A_1\\
\end{array}
\right),\quad
B=\left(
\begin{array}{lllll}
0&0&\cdots&0&\minus B_n\\
1&0&\cdots&0&\minus B_{n-1}\\
0&1&\cdots&0&\vdots\\
0&0&\ddots&0&\minus B_2\\
0&0&\cdots&1&\minus B_1\\
\end{array}
\right).}
\end{equation*}
\end{thm}

Note that $H(\alpha,\beta)$ can be conjugated into $\textrm{GL}_n(\mathbb Z)$ if and only if the polynomials $P(z)$ and $Q(z)$ above factor as cyclotomic polynomials, and so the roots of $P(z)$ and $Q(z)$ are reciprocal, meaning they are left invariant under the map $z\rightarrow z^{-1}$.  Since we are interested in integral monodromy groups, we will assume this throughout the article. Also, given a group $H(\alpha, \beta)$ as above, one obtains another hypergeometric group $H(\alpha',\beta')$ by taking $0\leq\alpha',\beta'<1$ to be $\alpha+d$ and $\beta+d$ modulo $1$, respectively.  The group $H(\alpha',\beta')$ is called a \emph{scalar shift} of $H(\alpha,\beta)$, and as pointed out in Remark~5.6 in \cite{bh}, we have that $H(\alpha,\beta)\cong H(\alpha',\beta')$, up to a finite center.  Thus when classifying all possible pairs $(\alpha,\beta)$ such that $H(\alpha,\beta)$ is integral and fixes a quadratic form of signature $(n-1,1)$ in Section~\ref{familyclassification} we will do so up to scalar shift.  Note that for any given $n$ there are only finitely many such groups $H(\alpha,\beta)$.

We also assume that $H$ is \emph{irreducible}, meaning that it fixes no proper subspace of $\mathbb C^n$, and \emph{primitive}, meaning that there is no direct sum decomposition $\mathbb C^n = V_1\oplus V_2\oplus\cdots\oplus V_k$ with $k>1$ and $\dim(V_i)\geq 1$ for all $1\leq i\leq k$ such that $H$ simply permutes the spaces $V_i$.  Finally, we define the reflection subgroup $H_r$ of $H$ to be the group generated by the reflections $\{A^kBA^{-k}\; |k\in \mathbb Z\}$.  By Theorem~5.3 of \cite{bh} we have that the primitivity of $H$ implies the irreducibility of $H_r$.
 With the notation above, Beukers-Heckman show in \cite{bh} that $H=H(\alpha,\beta)$ falls into one of three categories.  Given a hypergeometric monodromy group $H(\alpha,\beta)$, let
\begin{equation}\label{cratio}
c_{\alpha,\beta}:= A_n/B_n
\end{equation}
 where $A_n$ and $B_n$ are as in Theorem~\ref{levthm}.  Note that in the cases we consider $c_{\alpha,\beta}=\pm 1$.  Then $H(\alpha,\beta)$ belongs to one of the following categories.
\begin{itemize}
\item[(0)] A finite group (Beukers-Heckman list such cases completely in \cite{bh}, and we summarize these cases in Theorem~\ref{finiteH} below)
\item[(1)] If $n$ is even, $H$ is infinite, and $c_{\alpha,\beta}=1$ then $H\subset \textrm{Sp}_{n}(\mathbb Z)$ and $\textrm{Zcl}(H)=\textrm{Sp}_n(\mathbb C)$
\item[(2)] If $n$ is odd and $H$ is infinite; or if $n$ is even, $H$ is infinite, and $c_{\alpha,\beta}=-1$, then $H\subset \textrm{O}_{f_{\alpha,\beta}}(\mathbb Z)$ for some rational, unique up to scalar multiple quadratic form $f=f_{\alpha,\beta}$ in $n$ variables and $\textrm{Zcl}(H)=\textrm{O}_f(\mathbb C)$
\end{itemize}
As noted in the introduction, in this article we study hypergeometric groups $H$ which fall into category (2) and such that $H$ fixes a quadratic form of signature $(n-1,1)$.  We should mention that Beukers-Heckman show that the signature of $f_{\alpha,\beta}$ in category (2) is given by $(p,q)$ where $p+q=n$ and 
\begin{equation}\label{sig}
|p-q|=\left|\sum_{j=1}^n (-1)^{j+m_j}\right|
\end{equation}
where $m_j=|\{k\; |\; \beta_k<\alpha_j\}|$ and the $\alpha_i$ and $\beta_i$ are ordered as described in the introduction: $0\leq\alpha_j<1$, $0\leq \beta_j<1$. 

Although we concern ourselves with the case where $H$ is infinite, irreducible, and primitive, we state below Beukers-Heckman's classification of \emph{all} possible finite groups $H(\alpha,\beta)$.  In Section~\ref{familyclassification} we will need this classification to classify all of the primitive, irreducible hypergeometric groups $H(\alpha,\beta)$ which we consider -- i.e. integral orthogonal of signature $(n-1,1)$.  The precise statement of the theorem is taken from \cite{ms}.

\begin{thm}[\cite{bh},\cite{ms}]\label{finiteH} Let $\alpha=(\alpha_1,\dots,\alpha_n)\in \mathbb Q^n$ and $\beta=(\beta_1,\dots, \beta_n)\in\mathbb Q^n$ and let $H(\alpha,\beta)$ be as before.  Define $a_i:=e^{2\pi i \alpha_i}$ and $b_i=e^{2\pi i \beta_i}$ and let $\mathbf a:=\{a_1,\dots,a_n\}$ and $\mathbf b:=\{b_1,\dots, b_n\}$.  Let $P(z),Q(z)$ corresponding to $H(\alpha,\beta)$ be as in Theorem~\ref{levthm}.  If $H(\alpha,\beta)$ is finite then the corresponding polynomials $P(z),Q(z)$ are as in one of the cases below.

\noindent (1) The case where $H_r$ is primitive: in this case the corresponding polynomials $P(z)$ and $Q(z)$ belong to one of two infinite families or to one of $26$ sporadic examples.  One such infinite family corresponds to
$$P(z)=\frac{z^{n+1}-1}{z-1},\quad Q(z)=\frac{(z^j-1)(z^{n+1-j}-1)}{z-1}$$
where $n\geq 1$, $1\leq j\leq (n+1)/2$, $\textrm{gcd}(j,n+1)=1$.  The other infinite family is obtained by replacing $z$ with $-z$ above. 

\noindent (2) The case where $H_r$ acts reducibly on $\mathbb C^n$: in this case Theorem~5.3 of \cite{bh} implies that $H_r$ is primitive, and we have that there is some primitive $\ell$th root of unity $\zeta$ with $\ell>1$ such that $\zeta\mathbf a=\mathbf b$.   Then $(\alpha,\beta)$ gives the pair $P(z^{\ell}), Q(z^\ell)$ where $P$ and $Q$ are as in case (1) or correspond to one of the sporadic examples mentioned there.

\noindent (3) The case where $H$ is imprimitive and $H_r$ is irreducible: in this case the corresponding polynomials $P(z)$ and $Q(z)$ belong to one of two infinite families.  One such family is
$$P(z)=z^n+1,\quad Q(z)=(z^j-1)(z^{n-j}+1)$$
where $n\geq 3$, $1\leq j\leq n$, and $\textrm{gcd}(j,2n)=1$.  The other such family is obtained by replacing $z$ with $-z$ above.
\end{thm}

\subsection{Almost interlacing cyclotomic sequences}\label{interlace} In this section we classify all almost interlacing cyclotomic sequences (see Definition~\ref{almostinterlacing}).  This classification will be used in the next section to classify all hyperbolic hypergeometric monodromies in dimension $n>9$.
\begin{dfn} Two sequences $0 \le \alpha_1<\ldots < \alpha_n <1$ and $0 \le \beta_1<\ldots < \beta_n <1$ are called interlacing cyclotomic sequences if the following three conditions hold:
\begin{itemize}
\item[1.] $\alpha_i \ne \beta_j$ for every $1 \le i,j \le n$.
\item[2.] $\prod_{1 \le j \le n}(t-e^{2\pi i \alpha_j})$ and $\prod_{1 \le j \le n}(t-e^{2\pi  i \beta_j})$ are products of cyclotomic polynomials.
\item[3.] $\alpha_1<\beta_1<\alpha_2<\cdots<\alpha_n<\beta_n$ or $\beta_1<\alpha_1<\beta_2\cdots<\beta_n<\alpha_n$.
\end{itemize}
\end{dfn}
Define $r_i:=|\{j \mid \beta_j <\alpha_i\}|$ and $s_i:=|\{j \mid \alpha_j <\beta_i\}|$. 
Condition 3 in the above definition is equivalent to the following condition:
\begin{itemize}
\item[$3^*$.] $|\sum_{i=1}^n (-1)^{i+r_i}|=|\sum_{i=1}^n (-1)^{i+s_i}|=n$.
\end{itemize}

\begin{dfn}\label{almostinterlacing} Two sequences $0\le \alpha_1<\ldots < \alpha_n <1$ and $0\le \beta_1<\ldots < \beta_n <1$ are called almost interlacing cyclotomic sequences if the following three conditions hold:
\begin{itemize}
\item[1.] $\alpha_i \ne \beta_j$ for every $1 \le i,j \le n$.
\item[2.] $\prod_{1 \le j \le n}(t-e^{2\pi i \alpha_j})$ and $\prod_{1 \le j \le n}(t-e^{2\pi i \beta_j})$ are products of cyclotomic polynomials.
\item[3.] $|\sum_{i=1}^n (-1)^{i+r_i}|=|\sum_{i=1}^n (-1)^{i+s_i}|=n-2$.
\end{itemize}
\end{dfn}
Note that, by the expression for the signature in (\ref{sig}), it is precisely when $\alpha$ and $\beta$ are almost interlacing that $H(\alpha,\beta)$ is hyperbolic (signature $(n-1,1)$).  Beukers and Hekman classify all interlacing cyclotomic sequences in \cite{bh}.  We use their classification in order to classify all almost interlacing cyclotomic sequences. We start with the following lemma:
\begin{lemma}\label{interlacingtypes} Let $(\alpha_i)_{1 \le i \le n}$ and $(\beta_i)_{1 \le i \le n}$ be almost interlacing cyclotomic sequences with $n \ge 7$. Then $n$ is odd and $\alpha_{\frac{n+1}{2}}=\frac{1}{2}$ and $\beta_1=0$ or $\beta_{\frac{n+1}{2}}=\frac{1}{2}$ and $\alpha_0=0$. Moreover, if the first possibility happens then
one of the following 4 options holds:
\begin{itemize}
\item[(1)] The sequences $(c_i)_{1 \le i \le n}$ and $(d_i)_{1 \le i \le n}$ 
are interlacing cyclotomic sequences where:
$$\left\{\begin{array}{cc}
c_1 = \beta_1& \\
c_i = \alpha_{i-1}& 1 < i \le m \\
c_i = \alpha_i & m < i \le n \\
d_i = \beta_{i+1} & 1 \le i < m \\
d_m = \alpha_m & \\
d_i = \beta_i & m < i \le n
\end{array}\right.$$
\item[(2)] The sequences $(c_i)_{1 \le i \le n-1}$ and $(d_i)_{1 \le i \le n-1}$ 
are interlacing cyclotomic sequences where:
$$\left\{\begin{array}{cc}
c_1 = \beta_1 & \\
c_i = \alpha_{i-1}& 1 < i < m \\
c_i = \alpha_{i+1} & m \le i \le n-1 \\
d_i = \beta_{i+1} & 1 \le i \le n-1
\end{array}\right.$$
\item[(3)] The sequences $(c_i)_{1 \le i \le n-1}$ and $(d_i)_{1 \le i \le n-1}$ 
are interlacing cyclotomic sequences where:
$$\left\{\begin{array}{cc}
c_i = \alpha_i & 1 \le i < m\\
c_i = \alpha_{i+1}& m \le i \le n-1 \\
d_i = \beta_{i+2} & 1 \le i < m \\
d_m=\alpha_m& \\
d_i = \beta_{i+1} & m < i \le n-1
\end{array}\right.$$
\item[(4)] The sequences $(c_i)_{1 \le i \le n-2}$ and $(d_i)_{1 \le i \le n-2}$ 
are interlacing cyclotomic sequences where:
$$\left\{\begin{array}{cc}
c_i = \alpha_i & 1 \le i < m \\
c_i = \alpha_{i+2}& m \le i \le n-2 \\
d_i = \beta_{i+2} & 1 \le i \le n-2
\end{array}\right.$$
\end{itemize}
\end{lemma}
\begin{proof} The proof is divided into easy steps:
\begin{itemize}
\item[(a)] The sequences $(\alpha_i)_{1 \le i \le n}$ and $(\beta_i)_{1 \le i \le n}$ are cyclotomic so for every $t\in(0,1)$,
$$|\{j \mid \alpha_j=t\}|=|\{j \mid \alpha_j=1-t\}|$$ and $$|\{j \mid \beta_j=t\}|=|\{j \mid \beta_j=1-t\}|.$$
\item[(b)] For every $1 \le i \le n-1$ denote $\e_i:=(-1)^{|\{j \mid \beta_j\in (\alpha_i,\alpha_{i+1})\}|}$ and $\d_i:=(-1)^{|\{j \mid \alpha_j\in (\beta_i,\beta_{i+1})\}|}$. The equality 
$|\sum_{i=1}^n (-1)^{i+r_i}|=n-2$ implies that exactly
one of the following possibilities holds:
\begin{itemize}
\item[(\rmnum{1})] $\{i \mid \e_i=1\}=\{1\}$
\item[(\rmnum{2})] $\{i \mid \e_i=1\}=\{1,2\}$
\item[(\rmnum{3})] $\{i \mid \e_i=1\}=\{n-1\}$
\item[(\rmnum{4})] $\{i \mid \e_i=1\}=\{n-2,n-1\}$
\item[(\rmnum{5})] $\{i \mid \e_i=1\}=\{k-1,k\}$ for some $3 \le k \le n-2$.
\end{itemize} 
\item[(c)] The symmetry of step $(a)$ shows that options 
(\rmnum{3}) and (\rmnum{4}) are not possible and that:
\begin{itemize}
\item[$\circ$] If option 
(\rmnum{1}) happens then $\alpha_1=0$.
\item[$\circ$] If option 
(\rmnum{2}) happens then $\alpha_1=\alpha_2=0$. 
\item[$\circ$] If option (\rmnum{3})
happens then $n$ is odd, $k=\frac{n+1}{2}$ and $\alpha_k=\frac{1}{2}$. 
\end{itemize}
\item[(d)] Denote $m:=\frac{n+1}{2}$. Changing the roles of the $\alpha_i$'s and the $\beta_j$'s we 
see that $n$ is odd and that either $\alpha_m=\frac{1}{2}$ and $\beta_1=0$ or $\alpha_m=0$ and $\beta_1=\frac{1}{2}$.
\item[(e)]From now on we will assume that $\alpha_m=\frac{1}{2}$ so $\{i \mid \e_i=1\}=\{\frac{n-1}{2},\frac{n+1}{2}\}$ and
$\beta_1=0$.
\item[(f)] Since there are exactly $n$ $\beta_j$'s the assumptions of step (e) imply:
$$|\{j \mid \beta_j\in (\alpha_i,\alpha_{i+1})\}|=\left\{\begin{array}{cc}
0 & i= m-1 \text{ or } i=m\\
1 & \text{otherwise} 
\end{array}\right..$$ Furthermore, 
the symmetry of step (a) implies that if $\beta_2=0$ then also $\beta_3=0$ while if $\beta_2 \ne 0$ then $\beta_n=1-\beta_2>\alpha_n$.

\item[(g)] The symmetry of step (a) implies that either $a_{m-1}<a_m<a_{m+1}$ or $a_{m-1}=a_m=a_{m+1}$. Thus, there are only 4 options:
\begin{itemize}
\item[(1)] $$0 = \beta_1 < \beta_2 < \alpha_1 < \beta_3 < \alpha_2 < \cdots < \beta_m < \alpha_{m-1}<$$
$$  \alpha_m =\frac{1}{2}<
\alpha_{m+1} < \beta_{m+1} < \alpha_{m+2} < \cdots < \beta_{n-1} < \alpha_n < \beta_n$$
\item[(2)] $$0 = \beta_1 < \beta_2 < \alpha_1 < \beta_3 < \alpha_2 < \cdots < \beta_{m} < \alpha_{m-1}=$$
$$  \alpha_m =\frac{1}{2}=
\alpha_{m+1} < \beta_{m+1} < \alpha_{m+2} < \cdots < \beta_{n-1} < \alpha_n < \beta_n$$
\item[(3)] $$0 = \beta_1 = \beta_2=\beta_3 < \alpha_1 < \beta_4 < \alpha_2 < \cdots < \beta_{m+1} < \alpha_{m-1}<$$
$$  \alpha_m =\frac{1}{2}<
\alpha_{m+1} < \beta_{m+2} < \alpha_{m+2} < \cdots < \beta_{n} < \alpha_n$$
\item[(4)] $$0 = \beta_1 = \beta_2=\beta_3 < \alpha_1 < \beta_4 < \alpha_2 < \cdots < \beta_{m+1} < \alpha_{m-1}=$$
$$  \alpha_m =\frac{1}{2}=
\alpha_{m+1} < \beta_{m+2} < \alpha_{m+2} < \cdots < \beta_{n} < \alpha_n$$
\end{itemize}
\end{itemize}
\end{proof}

\subsection{Infinite families of hyperbolic monodromy groups}\label{familyclassification}

In this section we combine Lemma~\ref{interlacingtypes} with the classification of finite hypergeometric monodromy groups in Theorem~\ref{finiteH} to produce seven infinite families of primitive integer hypergeometric groups which fix a quadratic form of signature $(n-1,1)$.  Each family is either a two or three parameter family.  As will be clear in the proof of the following theorem, the $n$-dimensional two parameter families listed below are derived from the roots of 
\begin{equation}\label{polystate0}
P_m(z)= z^m+1 \quad \mbox{ and } \quad Q_m(z)= (z^j-1)(z^{m-j}+1),
\end{equation}
where $m$ is taken to be $n-2,n-1,$ or $n$, while the $n$-dimensional three parameter families are derived from the roots of
\begin{equation}\label{polystate}
P_{m,k}(z)= \frac{z^{\ell(k+1)}-1}{z^{\ell}-1} \quad \mbox{ and } \quad Q_{m,k}(z)= \frac{(z^{\ell j}-1)(z^{\ell(k+1-j)}-1)}{z^{\ell}-1}
\end{equation}
where $\ell k=m$ and $m$ is taken to be $n-2,n-1,$ or $n$.  These families describe all primitive hypergeometric hyperbolic monodromy groups in dimension $n>9$ up to scalar shift.  In addition to groups in these families, there are several sporadic groups in dimensions $n\leq 9$ which are listed in Table~\ref{table3} of Section~\ref{sporadic} along with all primitive hyperbolic $H(\alpha,\beta)$ in dimension $n\leq 9$.

\begin{thm}\label{theo}
Let $n\geq 1$ be odd and let $P_m,Q_m,P_{m,k},Q_{m,k}$ be as defined in (\ref{polystate0}) and (\ref{polystate}).  Let $\alpha=\{\alpha_1,\dots,\alpha_n\}$ and $\beta=\{\beta_1,\dots,\beta_n\}$ where $\alpha_i, \beta_i\in \mathbb Q$ for $1\leq i\leq n$ and let $H(\alpha,\beta)$ be as in Theorem~\ref{levthm}.  If $(\alpha,\beta)$ belongs to one of the following families or is a scalar shift of a pair in one of these families, then $H(\alpha,\beta)\subset\textrm{O}(n-1,1)$.  If $n>9$, this list of families completely describes (up to scalar shift) the groups $H(\alpha,\beta)\subset\textrm{O}_f(\mathbb Z)$ where $f$ is a quadratic form in $n$ variables of signature $(n-1,1)$.

\begin{itemize}
\item[1)] $\mathcal M_1(j,n)$:
\begin{align*}
\hspace{-1.1in} {\scriptstyle \alpha}&= {\scriptstyle \left(0, \frac{1}{2n}, \frac{3}{2n}, \dots, \frac{n-1}{2n}, \frac{n+1}{2n},\dots, \frac{2n-3}{2n}, \frac{2n-1}{2n}\right),}\\
\hspace{-1.1in} {\scriptstyle \beta}&= {\scriptstyle \left(\frac{1}{j},\dots, \frac{j-1}{j},\frac{1}{2}, \frac{1}{2n-2j}, \frac{3}{2n-2j}, \dots, \frac{2n-2j-3}{2n-2j}, \frac{2n-2j-1}{2n-2j}\right)}\nonumber
\end{align*}
where $0< j< n$ is an odd integer.
\item[2)] $\mathcal M_2(j,n)$:
\begin{align*}
{\scriptstyle \alpha}&={\scriptstyle \left(\frac{1}{2n-2}, \frac{3}{2n-2}, \dots, \frac{1}{2},\dots, \frac{2n-5}{2n-2}, \frac{2n-3}{2n-2}\right)},\\
{\scriptstyle \beta}&= {\scriptstyle \left(0,0,0, \frac{1}{j},\dots, \frac{j-1}{j}, \frac{1}{2n-2j-2}, \frac{3}{2n-2j-2}, \dots, \frac{n-j-3}{2n-2j-2}, \frac{n-j+1}{2n-2j-2}, \dots, \frac{2n-2j-5}{2n-2j-2}, \frac{2n-2j-3}{2n-2j-2}\right)}\nonumber
\end{align*}
where $0<j<n$ is an integer and $j/(n,j)$ is odd.
\item[3)] $\mathcal M_3(j,n)$:
\begin{align*}
\hspace{-.8in} \alpha&={\scriptstyle \left(\frac{1}{2n-4}, \frac{3}{2n-4}, \dots, \frac{1}{2},\frac{1}{2},\frac{1}{2},\dots, \frac{2n-7}{2n-4}, \frac{2n-5}{2n-4}\right)},\\
\hspace{-.8in} \beta&= {\scriptstyle \left(0,0,0, \frac{1}{j},\dots, \frac{j-1}{j}, \frac{1}{2n-2j-4}, \frac{3}{2n-2j-4}, \dots, \frac{2n-2j-7}{2n-2j-4}, \frac{2n-2j-5}{2n-2j-4}\right)}\nonumber
\end{align*}
where $0<j\leq n-2$ is an odd integer.
\item[4)] $\mathcal N_1(j,k,n)$:
\begin{align*}
\hspace{-2.3in} {\scriptstyle \alpha}&= (b_0,a_1,\dots,a_{n-1})\\
\hspace{-2.3in} {\scriptstyle \beta}&= (a_0,b_1,\dots, b_{n-1})\nonumber
\end{align*}
where $(j,k+1)=1$, $a_0=1/2$, $b_0=0$, $e^{2\pi i a_0 },\dots, e^{2\pi i a_{n-1}}$ are the roots of $P_{n,k}(z)$, and $e^{2\pi i b_0 },\dots, e^{2\pi i b_{n-1}}$ are the roots of $Q_{n,k}(z)$.

\item[5)] $\mathcal N_2(j,k,n)$:
\begin{align*}
\hspace{-2.1in} \alpha&=(0,a_0,\dots,a_{n-2}),\\
\hspace{-2.1in} \beta&= ({\scriptstyle \frac{1}{2}}, {\scriptstyle \frac{1}{2}}, {\scriptstyle \frac{1}{2}},b_2,\dots, b_{n-2}).\nonumber
\end{align*}
where $(j,k+1)=1$, $b_0=0$, $b_1=1/2$, $e^{2\pi i a_0 },\dots, e^{2\pi i a_{n-2}}$ are the roots of $P_{n-1,k}(z)$, and $e^{2\pi i b_0 },\dots, e^{2\pi i b_{n-2}}$ are the roots of $Q_{n-1,k}(z)$.
\item[6)] $\mathcal N_3(j,k,n)$:
\begin{align*}
\hspace{-2.1in}\alpha&=({\scriptstyle \frac{1}{2}},a_0,\dots,a_{n-2}),\\
\hspace{-2.1in}\beta&= (0,0,0, b_2,\dots, b_{n-2}).\nonumber
\end{align*}
where $(j,k+1)=1$, $b_0=0$, $b_1=1/2$, $e^{2\pi i a_0 },\dots, e^{2\pi i a_{n-2}}$ are the roots of $P_{n-1,k}(z)$, and $e^{2\pi i b_0 },\dots, e^{2\pi i b_{n-2}}$ are the roots of $Q_{n-1,k}(z)$.
\item[7)] $\mathcal N_4(j,k,n)$:
\begin{align*}
\hspace{-2.1in} \alpha&=({\scriptstyle \frac{1}{2},\frac{1}{2}}, a_0,\dots,a_{n-3}),\\
\hspace{-2.1in} \beta&= (0,0,b_0,b_1,\dots, b_{n-3}).\nonumber
\end{align*}
where $(j,k+1)=1$, $0<j<\frac{n-3}{2}$, $e^{2\pi i a_0 },\dots, e^{2\pi i a_{n-3}}$ are the roots of $P_{n-2,k}(z)$ and $e^{2\pi i b_0 },\dots, e^{2\pi i b_{n-3}}$ are the roots of $Q_{n-2,k}(z)$.
\end{itemize}
\end{thm}

\begin{proof}
  From the previous section, we have that every family of pairs $(\alpha,\beta)$ for which $H(\alpha,\beta)\subset \textrm{O}(n-1,1)$ can be derived from a family of pairs $(\alpha',\beta')$ for which $H(\alpha',\beta')$ is finite, but not necessarily primitive or irreducible.  There are several cases to consider.

\noindent{\bf Case 1:} $|\alpha'|=|\beta'|=n$.

Without loss of generality, assume $\alpha'_k=1/2$ for some $k$ and $\beta'_1=0$. By Lemma~\ref{interlacingtypes}, letting 
$$\alpha_j=\alpha'_j \mbox{ for $j\not=k$},$$ $$\alpha_k=0,$$  $$\beta_j=\beta'_j \mbox{ for $j\not=1$},$$ $$\beta_1=1/2$$  we have that $H(\alpha,\beta)\subset \textrm{O}(n-1,1)$.  

\vspace{.1in}

Suppose $H(\alpha',\beta')$ is finite, imprimitive, and irreducible.  From Theorem~\ref{finiteH} there are two infinite families of $(\alpha',\beta')$, the second of which gives pairs that are scalar shifts by $1/2$ of the pairs coming from the first family.  Since we are only interested in determining pairs up to scalar shift, we consider just one of these families, which has corresponding cyclotomic polynomials $P_n(z)$ and $Q_n(z)$ from (\ref{polystate0}) where $j$ is odd.  Up to scalar shift, this yields the family $\mathcal M_1$:
\begin{eqnarray*}\label{family2}
\alpha&=&{\scriptstyle \left(0, \frac{1}{2n}, \frac{3}{2n}, \dots, \frac{n-1}{2n}, \frac{n+1}{2n},\dots, \frac{2n-3}{2n}, \frac{2n-1}{2n}\right),}\\
\beta&=& {\scriptstyle \left(\frac{1}{j},\dots, \frac{j-1}{j},\frac{1}{2}, \frac{1}{2n-2j}, \frac{3}{2n-2j}, \dots, \frac{2n-2j-3}{2n-2j}, \frac{2n-2j-1}{2n-2j}\right)}\nonumber
\end{eqnarray*}
where $j$ is odd. Note that, unlike $(\alpha',\beta')$, the pair $(\alpha,\beta)$ is no longer imprimitive, and still irreducible.

\vspace{.1in}

Suppose $H(\alpha',\beta')$ is finite and reducible or primitive and irreducible.  Then for every $\ell|n$, Theorem~\ref{finiteH} gives two infinite families of $(\alpha',\beta')$, of which we consider just the first, as the other can be obtained via scalar shift of the first. The corresponding cyclotomic polynomials in this case are $P_{n,k}(z)$ and $Q_{n,k}(z)$ from (\ref{polystate}), where $j,k,\ell\in\mathbb N$, with $(j,k+1)=1$ and $\ell k=n$.   Let $a_0=1/2$ and denote the roots of $P_{n,k}(z)$ by $e^{2\pi i a_0 },\dots, e^{2\pi i a_{n-1}}$.  Let $b_0=0$ and denote the roots of $Q_{n,k}(z)$ by $e^{2\pi i b_0 },\dots, e^{2\pi i b_{n-1}}$.  Up to scalar shift, this yields the family $\mathcal N_1$:
\begin{eqnarray}\label{family3}
\alpha&=&(b_0, a_1,\dots,a_{n-1}),\\
\beta&=& (a_0,b_1,\dots, b_{n-1})\nonumber
\end{eqnarray}
Note that the pair $(\alpha,\beta)$ is primitive and irreducible.

\vspace{.1in}

\noindent{\bf Case 2:} $|\alpha'|=|\beta'|=n-1$.

As before, we consider the imprimitive irreducible, and the primitive irreducible or reducible cases.

\vspace{.1in}

Suppose $H(\alpha',\beta')$ is finite, imprimitive, and irreducible.  From Theorem~\ref{finiteH},  we have that up to scalar shift the only family of $(\alpha',\beta')$ in this case corresponds to the roots of $P_{n-1}(z)$ and $Q_{n-1}(z)$ in (\ref{polystate0}):
\begin{eqnarray*}
\alpha'&=&{\scriptstyle \left(\frac{1}{2n-2}, \frac{3}{2n-2}, \dots, \frac{2n-5}{2n-2}, \frac{2n-3}{2n-2}\right),}\\
\beta'&=& {\scriptstyle \left(0, \frac{1}{j},\dots, \frac{j-1}{j}, \frac{1}{2n-2j-2}, \frac{3}{2n-2j-2}, \dots, \frac{2n-2j-5}{2n-2j-2}, \frac{2n-2j-3}{2n-2j-2}\right)}\nonumber
\end{eqnarray*}
where $j/(n,j)$ is odd.  According to Lemma~\ref{interlacingtypes}, there are two ways to obtain from $(\alpha',\beta')$ a pair $(\alpha,\beta)$ for which $H(\alpha,\beta)$ is signature $(n-1,1)$ is obtained as above, giving the two families
\begin{eqnarray*}\label{family5}
\alpha&=&{\scriptstyle \left(\frac{1}{2n-2}, \frac{3}{2n-2}, \dots, \frac{1}{2},\dots, \frac{2n-5}{2n-2}, \frac{2n-3}{2n-2}\right),}\\
\beta&=& {\scriptstyle \left(0,0,0, \frac{1}{j},\dots, \frac{j-1}{j}, \frac{1}{2n-2j-2}, \frac{3}{2n-2j-2}, \dots, \frac{n-j-3}{2n-2j-2}, \frac{n-j+1}{2n-2j-2}, \dots, \frac{2n-2j-5}{2n-2j-2}, \frac{2n-2j-3}{2n-2j-2}\right)}\nonumber
\end{eqnarray*}
where $(j,2n)=1$ and
\begin{eqnarray*}\label{family5}
\alpha&=&{\scriptstyle \left(0, \frac{1}{2n-2}, \frac{3}{2n-2}, \dots, \frac{2n-5}{2n-2}, \frac{2n-3}{2n-2}\right),}\\
\beta&=& {\scriptstyle \left(\frac{1}{2},\frac{1}{2},\frac{1}{2}, \frac{1}{j},\dots, \frac{j-1}{j}, \frac{1}{2n-2j-2}, \frac{3}{2n-2j-2}, \dots, \frac{n-j-3}{2n-2j-2}, \frac{n-j+1}{2n-2j-2}, \dots, \frac{2n-2j-5}{2n-2j-2}, \frac{2n-2j-3}{2n-2j-2}\right)}\nonumber
\end{eqnarray*}
where $j/(n,j)$ is odd.  However, since these two families are shifts of each other by $1/2$, we record them under one family $\mathcal M_2$.

\vspace{.1in}

Suppose $H(\alpha',\beta')$ is finite and reducible or primitive irreducible.  From Theorem~\ref{finiteH}, we have that for every $\ell|n-1$ that there is only one family up to scalar shift of $(\alpha',\beta')$ for which $H(\alpha',\beta')$ is reducible.  Namely, let $j,k,\ell\in\mathbb N$, with $(j,k+1)=1$ and $\ell k=n-1$.  Denote the roots of $P_{n-1,k}(z)$ in (\ref{polystate}) by $e^{2\pi i a_0 },\dots, e^{2\pi i a_{n-2}}$.  Let $b_0=0$, $b_1=1/2$, and denote the roots of $Q_{n-1,k}(z)$ in (\ref{polystate}) by $e^{2\pi i b_0 },\dots, e^{2\pi i b_{n-2}}$. Then the family
\begin{eqnarray*}
\alpha'&=&(a_0,\dots,a_{n-2}),\\
\beta'&=& (b_0,\dots, b_{n-2})\nonumber
\end{eqnarray*}
is the only family up to scalar shift in this case such that $H(\alpha',\beta')$ is reducible.  According to Lemma~\ref{interlacingtypes}, there are two ways to obtain a pair $(\alpha,\beta)$ for which $H(\alpha,\beta)$ is signature $(n-1,1)$.  One way is to remove the term $b_0=0$ and add two $1/2$'s to $\beta'$, and add the term $0$ to $\alpha'$, yielding the family $\mathcal N_2$:

\begin{eqnarray}\label{family6}
\alpha&=&(0,a_0,\dots,a_{n-2}),\\
\beta&=& ({\scriptstyle \frac{1}{2},\frac{1}{2}}, b_1,\dots, b_{n-2}).\nonumber
\end{eqnarray}
\vspace{.1in}

Another way is to add the term $1/2$ to $\alpha'$, take away the term $b_1=1/2$ from $\beta'$, and insert two $0's$ into $\beta'$, yielding the family $\mathcal N_3$:
\begin{eqnarray}\label{family4}
\alpha&=&({\scriptstyle \frac{1}{2}},a_0,\dots,a_{n-2}),\\
\beta&=& (0,0,0, b_2,\dots, b_{n-2}).\nonumber
\end{eqnarray}

\noindent{\bf Case 3:} $|\alpha'|=|\beta'|=n-2$.

\vspace{0.1in}

Suppose $H(\alpha',\beta')$ is finite, imprimitive, and irreducible.  From Theorem~\ref{finiteH},  we have that up to scalar shift the only family of $(\alpha',\beta')$ in this case is obtained from the roots of the cyclotomic polynomials $P_{n-2}(z)$ and $Q_{n-2}(z)$ in (\ref{polystate0}):
\begin{eqnarray*}
\alpha'&=&{\scriptstyle \left(\frac{1}{2n-4}, \frac{3}{2n-4}, \dots, \frac{2n-7}{2n-4}, \frac{2n-5}{2n-4}\right),}\\
\beta'&=& {\scriptstyle \left(0, \frac{1}{j},\dots, \frac{j-1}{j}, \frac{1}{2n-2j-4}, \frac{3}{2n-2j-4}, \dots, \frac{2n-2j-7}{2n-2j-4}, \frac{2n-2j-5}{2n-2j-4}\right)}\nonumber
\end{eqnarray*}
where $j$ is odd.  The corresponding $(\alpha,\beta)$ for which $H(\alpha,\beta)$ is signature $(n-1,1)$ is obtained by adding two $1/2$'s to $\alpha'$, and adding two $0$'s to $\beta'$, yielding the family $\mathcal M_3$:
\begin{eqnarray*}\label{family8}
\alpha&=&{\scriptstyle \left(\frac{1}{2n-4}, \frac{3}{2n-4}, \dots, \frac{1}{2},\frac{1}{2},\frac{1}{2},\dots, \frac{2n-5}{2n-2}, \frac{2n-3}{2n-2}\right)},\\
\beta&=& {\scriptstyle \left(0,0,0, \frac{1}{j},\dots, \frac{j-1}{j}, \frac{1}{2n-2j-4}, \frac{3}{2n-2j-4}, \dots, \frac{2n-2j-7}{2n-2j-4}, \frac{2n-2j-5}{2n-2j-4}\right)}\nonumber
\end{eqnarray*}
where $j$ is odd.

\vspace{.1in}

Suppose $H(\alpha',\beta')$ is finite and reducible or primitive irreducible.  From Theorem~\ref{finiteH}, we have that for every $\ell|n-2$ that there is only one family up to scalar shift of $(\alpha',\beta')$ for which $H(\alpha',\beta')$ is reducible.  Namely, let $j,k,\ell\in\mathbb N$, with $(j,k+1)=1$ and $\ell k=n-2$.  Denote the roots of $P_{n-2,k}(z)$ in (\ref{polystate}) by $e^{2\pi i a_0 },\dots, e^{2\pi i a_{n-3}}$.  Denote the roots of $Q_{n-2,k}(z)$  in (\ref{polystate}) by $e^{2\pi i b_0 },\dots, e^{2\pi i b_{n-3}}$. Then the family
\begin{eqnarray*}
\alpha'&=&(a_0,\dots,a_{n-3}),\\
\beta'&=& (b_0,\dots, b_{n-3})\nonumber
\end{eqnarray*}
is the only family up to scalar shift in this case such that $H(\alpha',\beta')$ is reducible. 

The corresponding $(\alpha,\beta)$ for which $H(\alpha,\beta)$ is signature $(n-1,1)$ is obtained by adding two $1/2$'s to $\alpha'$, and adding two $0$'s to $\beta'$, yielding the family $\mathcal N_4$:

\begin{eqnarray}\label{family9}
\alpha&=&({\scriptstyle \frac{1}{2},\frac{1}{2}}, a_0,\dots,a_{n-3}),\\
\beta&=& (0,0,b_0,b_1,\dots, b_{n-3}).\nonumber
\end{eqnarray}

\vspace{0.1in}

This exhausts all of the possibilities and thus we have the statement in the theorem as desired.  Since from the previous section we have that any primitive hypergeometric hyperbolic monodromy group is obtained by permuting the coordinates of some interlacing pair $(\alpha',\beta')$, and since there are no sporadic such interlacing pairs in dimensions $n>7$ , we have that the families $\mathcal M_1,\mathcal M_2,\mathcal M_3, \mathcal N_1,\mathcal N_2,\mathcal N_3,$ and $\mathcal N_4$ completely describe the primitive hyperbolic hypergeometric $H(\alpha,\beta)$'s in dimension $n>9$.
\end{proof}

\begin{remark}
We end this subsection with a remark about the integrality of the coefficients of $_nF_{n-1}$ for hyperbolic hypergeometrics.  Writing $R(t)=P(t)/Q(t)$, where $P$ and $Q$ are as in Theorem~\ref{levthm}, as
$$\frac{(t^{a_1}-1)(t^{a_2}-1)\cdots(t^{a_K}-1)}{(t^{b_1}-1)(t^{b_2}-1)\cdots(t^{b_L}-1)}$$
with $a_j,b_k$ positive integers and $a_j\not=b_k$.  The coefficients of the corresponding hypergeometric function are

\begin{equation}\label{star}
   u_m=\frac{(ma_1)!(ma_2)!\cdots(ma_K)!}{(mb_1)!\cdots(mb_L)!} \tag{$\ast$}
  \end{equation}
  If these are to be integers then $d:=L-K$ must be at least $1$. Note that $d$ is the multiplicity of the $k$'s for which $\beta_k=0$.  Then
  $$\sum_{m=0}^\infty u_mz^m=_nF_{n-1}(\alpha_1,\dots,\alpha_n; \beta_1,\dots,\beta_{n-1}|Cz)$$
where $C=\frac{a_1^{a_1}a_2^{a_2}\cdots a_K^{a_K}}{b_1^{b_1}b_2^{b_2}\cdots b_L^{b_L}}$ (see \cite{vill}, \cite{bob}).

It is easy to check using Landau's criterion \cite{land} and our characterization of hyperbolic hypergeometric monodromies, that for the latter $u_m\in \mathbb Z$ iff $d=3$.  From the description of the families of hyperbolic hypergeometric monodromies in Theorem~\ref{theo} we see that $M_2(j,n)$, $M_3(j,n)$, $N_3(j,k,n)$, and $N_4(j,k,n)$ are the infinite families with integral coefficients while the sporadic ones can be read off from Tables 2 and 3, namely the entries with $\beta_1=\beta_2=\beta_3=0$.  This can be used to give the list of integral factorials (\ref{star}) which correspond to hyperbolic hypergeometric monodromies.  This is the analogue of \cite{vill} and \cite{bob} who classify the integral $u_m$'s in (\ref{star}) which correspond to finite monodromy groups (which in turn correspond to $d=1$).
\end{remark}

\subsection{The quadratic form}\label{quadform}
In this section we calculate the quadratic forms preserved by the primitive hyperbolic hypergeometric monodromy groups described in the previous section.  This will be a necessary ingredient in our minimal distance graph method described in Section~\ref{mindistgraph}.  Recall that the monodromy group $H=H(\alpha,\beta)$ is generated by two matrices
$$
A:=\left(\begin{array}{ccccc}
0 & 0 & \cdots & 0 &  -1\\
1 & 0 & \cdots & 0 & -a_{n-1} \\
0 & 1 & \cdots & 0 & -a_{n-2} \\
\vdots & \vdots & \ddots &\vdots &\vdots \\
0 & 0 & \cdots & 1 & -a_{1}
\end{array}\right)
\text{ and }
B:=\left(\begin{array}{ccccc}
0 & 0 & \cdots & 0 &  1\\
1 & 0 & \cdots & 0 & -b_{n-1} \\
0 & 1 & \cdots & 0 & -b_{n-2} \\
\vdots& \vdots& \ddots &\vdots & \vdots\\
0 & 0 & \cdots & 1 & -b_{1}
\end{array}\right)
$$
where the characteristic polynomials of $A$ and $B$,
$P(x):=x^n+a_1x^{n-1}+\cdots +a_{n-1}x+1$
and
$Q(x):=x^n+b_1x^{n-1}+\cdots +b_{n-1}x-1$ respectively, 
are products of cyclotomic polynomials.  Denote
\begin{equation}\label{Candv}
C:=A^{-1}B=\left(\begin{array}{ccccc}
1     & 0 &   \cdots & 0 & -(a_{n-1}+b_{n-1}) \\
0     & 1 &  \cdots  & 0 & -(a_{n-2}+b_{n-2}) \\
  \vdots       &  \vdots         &\ddots    &   \vdots      &  \vdots            \\
0     & 0 &     \cdots      & 1 &  -(a_1+b_1)        \\
0     & 0 &   \cdots  & 0 & -1
\end{array}\right) \text{  and  } v:=\left(\begin{array}{c}
a_{n-1}+b_{n-1} \\
a_{n-2}+b_{n-2} \\
\vdots\\
a_1+b_1        \\
2
\end{array}\right).
\end{equation}
The eigenvalues of $C$ are $1$ and $-1$ and their 
geometric multiplicities are $n-1$ and $1$ respectively. 
The first $n-1$ elements $e_1,\ldots,e_{n-1}$ of the standard basis
are eigenvectors for the eigenvalue $1$ while $v$ is an eigenvector
for the eigenvalue $-1$.

\begin{lemma}\label{lemtec1} Let $f \in \M_{n \times n}(\Z)$. Then $A^tfA=f$ if and only if $\tilde{A}^{i-1}fe_1=fe_i$ for every $2 \le i \le n$ where
$\tilde{A}=(A^t)^{-1}$.
\end{lemma}
\begin{proof} The only if part follows from comparing 
the first $n-1$ columns of both sides of 
$\tilde{A}f=fA$. For the if part it is suffices to check that also the last columns of $\tilde{A}f$ and $fA$ are equal. The equalities of the first $n-1$ columns imply 
that $\tilde{A}^{i}fe_1=fA^ie_1$ for every $1 \le i \le n-1$. The characteristic polynomial $P(x)$ of $A$ is also the characteristic polynomial of $\tilde{A}$ since it is a a product of cyclotomic polynomials. Thus,
$$ fAe_n=fA^{n}e_1=-f(a_1A^{n-1}+\cdots +a_{n-1}A+A)e_1=
$$ 
$$-(a_1\tilde{A}^{n-1}+\cdots +a_{n-1}\tilde{A}+\tilde{A})fe_1=
\tilde{A}^nfe_1=\tilde{A}fA^{n-1}e_1=\tilde{A}fe_n
$$
\end{proof}
\begin{lemma}\label{lemtec2} Assume that $f \in \M_{n \times n}(\Z)$ satisfies $A^tfA=f$. Then $B^tfB=f$ if and only if $fe_1$ is Euclidean-orthogonal to ${B}^{2-i}v$ for every $2 \le i \le n$. In particular, if $A^tfA=f$ and $B^tfB=f$ then $fe_i$ is Euclidean-orthogonal to $v$ for every $1 \le i \le n-1$. 
\end{lemma}
\begin{proof} 
Lemma \ref{lemtec1} and its analog with respect to $B$ implies that $B^tfB=f$ if and only if  $\tilde{A}^{i-1}fe_1=\tilde{B}^{i-1}fe_1$ for every $2 \le i \le n$. As the first row of
$\tilde{A}-\tilde{B}$ is $-v^t$ while the other rows are zeros, $fe_1$ must be orthogonal to $v$ (under the usual Euclidean scalar product). For every $i \ge 3$
$$\tilde{A}^{i-1}-\tilde{B}^{i-1}=\tilde{A}\left(\tilde{A}^{i-2}
-\tilde{B}^{i-2}\right)-
\left(\tilde{B}-\tilde{A}\right)\tilde{B}^{i-2}$$
so by induction on $i$ we see that $f$ satisfies the required condition if and only if
$(\tilde{B}-\tilde{A})\tilde{B}^{i-2}f=0$ which is equivalent to $f$ being Euclidean-orthogonal to $B^{2-i}v$ for every $3
\le i \le n$. Finally, if $A^tfA=f$, $B^tfB=f$ then also 
$C^tfC=C$. Thus, for $1 \le i \le n-1$ we have $v^tfe_i=(Cv)^tfCe_i=-v^tfe_i$ so $v^tfe_i=0$.
\end{proof}

\begin{prop}\label{prop form} Let $A$, $B$, $C$, $P(X)$, $Q(X)$ and $v$ be as above. Assume that the group $H:=\langle A,B\rangle$  is primitive. Then:
\begin{itemize}
\item[(1)] The vectors $v,Bv,\ldots B^{n-1}v$ are linearly independent and $H=\langle A,B\rangle$ preserves the lattice $L$ spanned by them.
\item[(2)] There exists a unique (up to a scalar product) non-zero integral quadratic form $(\cdot,\cdot)$ such that $A$ and $B$ belong to its orthogonal group.
\item[(3)] If the quadratic form is normalized to have 
$(v,v)=-2$ then $(v,u)$ equals to minus the $n^{\text{th}}$-coordinate of $u$ for every $u \in \Z^n$.
\item[(4)] If the coefficients of $Q(x)$ satisfy  $B_i=(-1)^{i}$ then the matrix representing $f$ w.r.t $v,Bv,\ldots B^{n-1}v$ is given by 
$$f_{i,j}:=\left\{\begin{array}{ccc}
-2 &\text{if}& |i-j|=0 \\
-1-a_1 & \text{if}&|i-j|=1 \\
-a_k-a_{k-1}& \text{if}& |i-j|=k \not \in \{0,1\}
\end{array}\right.$$ 
\end{itemize}
\end{prop}
\begin{proof}
We start by proving the existence part of $(2)$.
By Lemma \ref{lemtec2} there exists a non-zero $f \in \M_{n \times n}(\Z)$ such that $A^tfA=A$ and $B^tfB=f$.  Assume first that $f$ is anti-symmetric so $v^tfv=0$. Lemma \ref{lemtec2} implies that $v^tfe_i=0$ for every $1 \le i \le n-1$. 
Since $v,e_1,\ldots,e_{n-1}$ span $\Q^n$, the set
$\{w \in \Q^n \mid w^tfu=0 \text{ for all } u\in \Q^n\}$  is a non-trivial proper subspace of $\Q^n$. This subspace is preserved by $H$, a contradiction to the irreducibility of $H$. Thus, $f$ is not anti-symmetric and the quadratic form $(u_1,u_2):={u_1}^t(f+f^t)u_2$ has the desired properties. 
Lemma \ref{lemtec2} implies that $v^t(f+f^t)=\left(\begin{array}{cccc}
0 & \cdots & 0 & c
\end{array}\right)$ for some $c \ne 0$ so $(3)$ holds.

Since $H$ is primitive, every non-trivial normal subgroup of it is irreducible. Thus in order to prove $(1)$ it is enough to show that the normal subgroup $H_r:=\langle B^{-i}CB^i\mid i \in \mathbb Z\rangle$ preserves the $\mathbb Z$-lattice spanned by $v,Bv,\ldots, B^{n-1}v$.  Property $(3)$ implies that $Cu=u-2\frac{(v,u)}{(v,v)}v$ for every $u \in \Z^n$.
Since $B$ preserves the quadratic form,
$B^{-i}CB^{i}u=u-2\frac{(B^{-i}v,u)}{(B^{-i}v,B^{-i}v)}B^{-i}v$ for every $u \in \Z^n$.
Thus, $H_r$ preserves the $\Z$-lattice spanned by $v,Bv, 
\ldots, B^{n-1}v$ and the proof of $(1)$ is complete. The uniqueness part of $(2)$ now follows from $(1)$ together with Lemma \ref{lemtec2}.

For $1 \le i \le n-1$ the $(n-i)$-th and the $(n-i+1)$-th coordinates of bottom row of $B^i$ are equal to 1 and all the other coordinates of this row are zero. Thus, the last coordinate of $B^iv$ is $a_1+1$ if $i=1$ and $a_{i}+a_{i-1}$ if $2 \le i \le n-1$. From $(3)$ we get $(v,B^iv)=f_{1,i}$ for every $1 \le i \le n$. Property $(4)$ follows
from the fact the $f_{i,j}$ depends only on $|i-j|$ since $B$ preserves the quadratic form.
\end{proof}
We now record the invariant quadratic form for two cases that we consider in Section~\ref{thinguys}.
\begin{corol}\label{corol form 1} Let $(
\cdot,\cdot)$ be the normalized quadratic form preserved by $\mathcal{N}_1(1,n,n)$. If $1 \le i,j \le n$ then $$(B^iv,B^jv)=
\left\{\begin{array}{ccc}
-2 & \text{ if }& |i-j|=0\\
-3 & \text{ if }& |i-j|=1\\
-4 & \text{ if }& |i-j|\ge 2
\end{array}\right.$$
\end{corol}
\begin{corol}\label{corol form 3} Let $(
\cdot,\cdot)$ be the normalized quadratic form preserved by $\mathcal{N}_1(3,n,n)$. If $1 \le i,j \le n$ then 
$$(B^iv,B^jv)=\left\{\begin{array}{ccc}
-2 &\text{if}& |i-j|=0 \\
-4 & \text{if}&|i-j|=1 \\
-8 & \text{if}& |i-j|=2 \text { or } |i-j|=n-1\\
-11 & \text{if}& |i-j|=3 \text { or }  |i-j|=n-2\\
-12 & \text{otherwise} 
\end{array}\right.$$
\end{corol}

\section{Cartan involutions}\label{cartaninvol}

\subsection{Hyperbolic reflection groups}\label{cartan1}

As noted in the Introduction we make crucial use of Vinberg's \cite{vin}, \cite{vin2} as well as Nikulin's \cite{nik} results concerning the size of groups generated by hyperbolic reflections.  Before reviewing this arithmetic theory we begin with quadratic forms and hyperbolic space over the reals.  A non-degenerate real quadratic form $f(x_1,\dots,x_n)$ with corresponding symmetric bilinear form $(\;,\;)$ is determined by its signature.  Our interest is in the case that $f$ has signature $(-1,1,\dots,1)$, which after a real linear change of variable can be brought to the form
\begin{equation}\label{eq3.1}
f(x_1,\dots,x_n)=-x_1^2+x_2^2+\cdots+x_n^2.
\end{equation}
The null cone $C$ of $f$ is $\{x\;|\; (x,x)=0\}$ as depicted in Figure~\ref{conefig}.

\begin{figure}[htp]
\centering
\includegraphics[height=50mm]{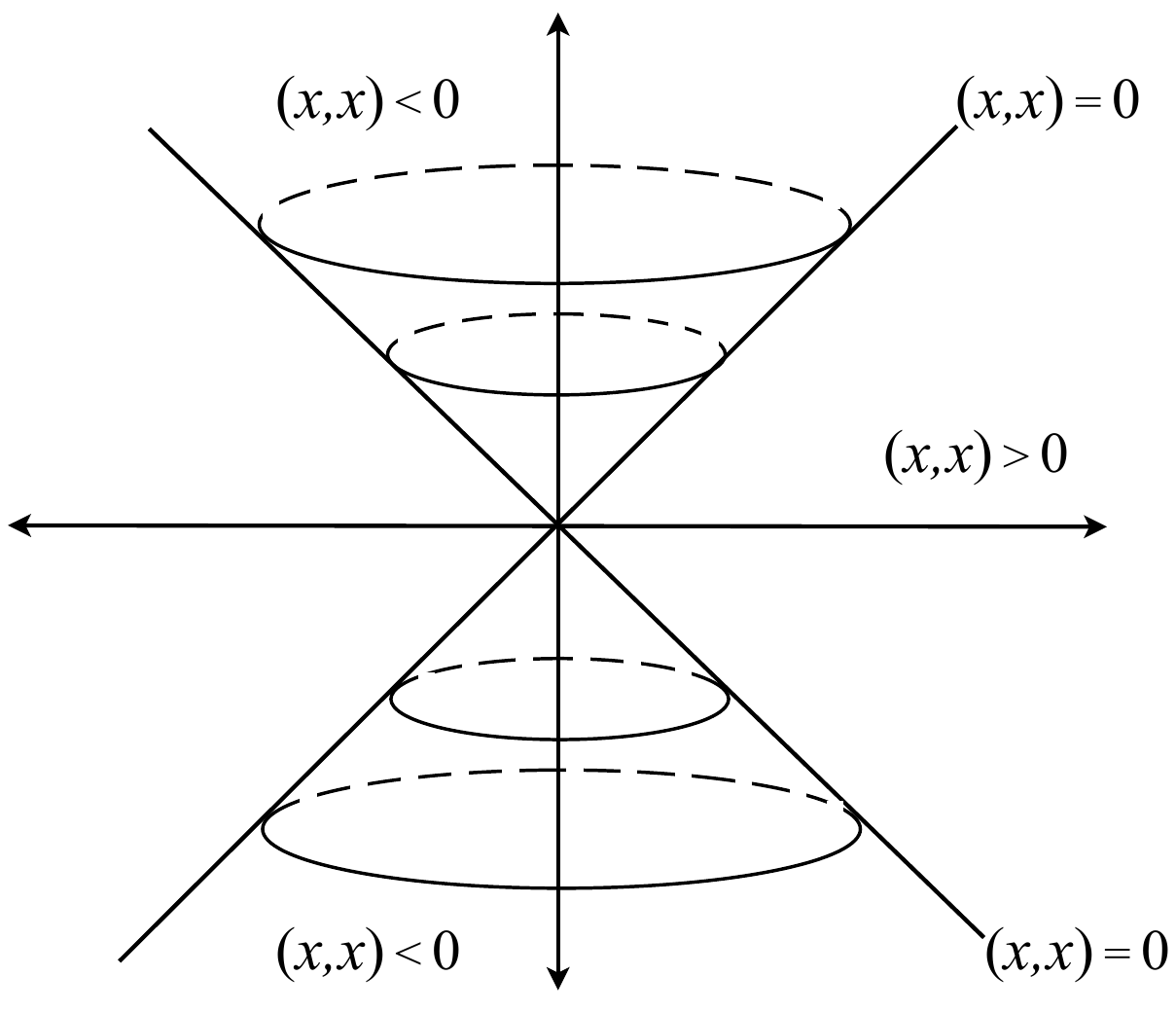}
\caption{}\label{conefig}
\end{figure}
$\mathbb R^n\backslash C$ consists of $3$ components, the outside, $(x,x)>0$, which consists of one component, and the inside, $(x,x)<0$, which consists of $2$ components: $x_1>0$ and $x_1<0$.  The quadrics $(x,x)=k$ where $k<0$ are two sheeted hyperboloids and either sheet, say the one with $x_1>0$ and $k=-2$, can be chosen as a model of hyperbolic $n-1$ dimensional space (we assume that $n\geq 3$), $\mathbb H^{n-1}$.  This is done by restricting the line element $ds^2=-dx_1^2+dx_2^2+\cdots+dx_n^2$ to $\mathbb H^{n-1}$.  The quadrics $(x,x)=k$ with $k>0$ are one sheeted hyperboloids.  If $g\in \textrm{O}_f(\mathbb R)$, the real orthogonal group of $f$, then $g$ preserves the quadrics, however it may switch the sheets of the two sheeted hyperboloid.  If it preserves each of these, then $g$ acts on $\mathbb H^{n-1}$ isometrically, while if $g$ switches the sheets then $-g$ acts isometrically on $\mathbb H^{n-1}$.  In either case we obtain an induced isometry.  Of particular interest to us are linear reflections of $\mathbb R^n$ given by $v$'s with $(v,v)\not=0$ (i.e. not on the null cone).  For such a $v$ denote by $r_v$ the linear transformation
\begin{equation}\label{eq3.2}
r_v(y)=y-\frac{2(v,y)}{(v,v)}v.
\end{equation}
It is easy to see that $r_v\in \textrm{O}_f$, that it is an involution, and that it preserves the components of the complement of the light cone if $(v,v)>0$, and switches them if $(v,v)<0$.  In the first case, the induced isometry of $\mathbb H^{n-1}$ is a hyperbolic reflection.  The fixed point set of $r_v$ in $\mathbb R^n$ is $v^\bot=\{y\; |\; (v,y)=0\}$ and this intersects $\mathbb H^{n-1}$ in an $n-2$-dimensional hyperbolic hyperplane and the induced isometry is a (hyperbolic) reflection about this hyperplane.  In the second case, $(v,v)<0$, the induced isometry $-r_v$ fixes the line $\ell_r=\{\lambda v\; |\; \lambda\in\mathbb R\}$ and this line meets $\mathbb H^{n-1}$ in a point $p_v$.  The induced isometry of $\mathbb H^{n-1}$ in this case inverts geodesics in $p_v$ and is a Cartan involution.  These two involutive isometrics of $\mathbb H^{n-1}$ are quite different.  In the case of a discrete group of motions generated by hyperbolic reflections, there is a canonical fundamental domain, namely the connected components of the complement of the hyperplanes corresponding to all the reflections.  This is the starting point to Vinberg's theory of reflective groups.  For groups generated by Cartan involutions there is no apparent geometric approach.

We turn to the arithmetic theory.  Let $L\subset \mathbb Q^n$ be a quadratic lattice, that is a rank $n$ $\mathbb Z$-module equipped with a nondegenerate integral symmetric bilinear form $(\;,\;)$, whose associated quadratic form $f(x)=(x,x)$ is rational.  The dual lattice $L^\ast$ of $L$ is given by $$L^{\ast}=\{x\in\mathbb Q^n\; |\; (x,y)\in \mathbb Z \mbox{ for all } y\in L\}.$$  It contains $L$ as a finite sublattice and the invariant factors of $L$ are defined to be the invariants of the torsion module $L^\ast/L$.  The product of these invariants is equal to $\pm d(L)$, where $d(L)$ is the discriminant of $L$ (that is $\det (F)$ where $F$ is an integral matrix realization of $f$).  We assume that $L\otimes \mathbb R$ is hyperbolic and of signature $(-1,1,\dots,1)$.  The $\mathbb Z$ analogue of the null cone is $C(L)=\{x\in L\;| \; (x,x)=0\}$ and that of the quadrics is $V_k(L):=\{x\in L\;|\; (x,x)=k\}$ for $k\in \mathbb Z$, $k\not=0$.  We assume throughout that $f$ is isotropic over $\mathbb Q$ which means that $C(L)\not=\varnothing$ (and is in fact infinite).  The group of integral automorphs of $L$ is denoted by $\textrm{O}(L)$ or $\textrm{O}_f(\mathbb Z)$.  A primitive vector $v$ in $V_k(L)$ is called a $k$-root if $\frac{2v}{k}\in L^\ast$.  In this case the linear reflection $r_v$ of $\mathbb Q^n$ given by (\ref{eq3.2}) is integral and lies in $\textrm{O}(L)$.  For $k=\pm 1$ or $\pm 2$, $\frac{2v}{k}$ is always in $L^\ast$ and we will have occasions where $k=4$ yields root vectors as well.  As discussed above, the root vectors with $k>0$ induce hyperbolic reflections on $\mathbb H^{n-1}_f$ (which we choose to be one of the components of $V_{-2}(\mathbb R)=\{x\in L\otimes\mathbb R\; |\; (x,x)=-2\}$) while for $k<0$ they induce Cartan involutions.

The Vinberg reflective group $R(L)$ is the subgroup of $\textrm{O}(L)$ generated by all hyperbolic reflections coming from root vectors with $k>0$.  It is plainly a normal subgroup of $\textrm{O}(L)$ and the main result in \cite{vin2} asserts that if $n\geq 30$ then $|\textrm{O}(L)/R(L)|=\infty$.  Vinberg's proof is based on studying the fundamental domain in $\mathbb H^{n-1}$ of $R(L)$ and relating it to the cusps of $\textrm{O}(L)\backslash \mathbb H^{n-1}$ which correspond to null vectors $w\in C(L)$.  In particular it uses the assumption that $f$ is isotropic.

Nikulin's results \cite{nik} are concerned with the case that $L$ is even, that is $(x,x)\in2\mathbb Z$ for all $x\in L$, and the subgroup $R_2(L)$ generated by all the $2$-root vectors in $L$ (note he chooses $f$ to have signature $(1,-1,-1,\dots,-1)$ so that his $-2$-root vectors are our $2$-root vectors).  $R_2(L)$ is a subgroup of $R(L)$ and it is also normal in $\textrm{O}(L)$.  For these he gives a complete classification of all $L$'s for which $\textrm{O}(L)/R_2(L)$ is finite.  There are only finitely many such and for $n\geq 5$ the list is quite short.  $L$ is called two elementary if $L^\ast/L$ is $(\mathbb Z/2\mathbb Z)^a$ for some $a$.  If $L$ is not two elementary and $n\geq 5$ and odd (the last condition on $n$ is what is of interest to us for our applications) then $|\textrm{O}(L)/R_2(L)|=\infty$ unless $L$ is isomorphic to $U\oplus K$ and $K$ is one of

$$A_3, A_1\oplus A_2, A_1\oplus A_2^2, A_1^2\oplus A_3, A_2\oplus A_3, A_1\oplus A_4,$$ 

\vspace{-0.2in}

$$A_5, D_5, A_7, A_3\oplus D_4, A_2\oplus D_5, D_7, A_1\oplus E_8, A_3\oplus E_8$$

\noindent or $L$ is isomorphic to 
$$U(4)\oplus A_1^3, \langle -2^k\rangle\oplus D_4 \mbox{ for } k=2,3,4, \mbox{ or } \langle-2,3\rangle \oplus A_2^2.$$

Here the positive definite ($2$-reflective) lattices $A_n$, $D_n$, $E_n$ are the standard ones (see \cite{cs}) while $U=\left[\begin{array}{cc}0&1\\1&0\end{array}\right]$ and $U(4)=\left[\begin{array}{cc}0&4\\4&0\end{array}\right]$.

As a consequence unless $L$ is two elementary or its invariant factors are among the list below, $|\textrm{O}(L)/R_2(L)|=\infty.$

\begin{table}[htp]
\centering
    \begin{tabular}{| l | p{6cm} |}
    \hline
    Dimension & Factors\\ \hline
    $5$ & $2\cdot 3, 4, 2\cdot 3^3, 4^2\cdot 2^3, 4\cdot 4, 4\cdot 8, 4\cdot 16$ \\ \hline
    $7$ & $2\cdot 3\cdot 3, 2\cdot 2\cdot 4, 3\cdot 4, 2\cdot 5, 3\cdot 2, 4$ \\ \hline
    $9$ & $8, 4\cdot 4, 3\cdot 4, 4, 3\cdot 3$ \\ \hline
    $11$ & none \\ \hline
    $13$ & $4$ \\ \hline
    $\geq 15$ & none \\ 
    \hline
    \end{tabular}
\caption{}
\label{table2}
\end{table}

\subsection{Pseudo-reflections and Cartan involutions}

The almost interlacing condition on the $(\alpha,\beta)$'s ensures that $H(\alpha,\beta)$ lies in an orthogonal group of a quadratic form $f$ of signature $(p,q)$, i.e. $(\stackrel{p}{1,1,\dots,1},\stackrel{q}{-1,-1,\dots,-1})$ with $|p-q|=n-2$.  The form $f$ is unique up to a scalar multiple and we normalize $f$ so that $f(v,v)=-2$ where $v$ is the vector given in (\ref{Candv}).  In order to determine what type of involution these critical pseudo-reflections $r_v$ induce on hyperbolic space we need to determine the number $p-q$, when $f$ is so normalized.  This can be done by keeping track of the signs in the calculations in Proposition~4.4 and Theorem~4.5 in \cite{bh}.  With $u=v$ in equation (4.5) in \cite{bh} (note there is a misprint there: it should be $D(x)$ rather than $D(u)$), $\eta=1$, and $c=-1$ we have 
\begin{equation}\label{eq1}
D(x)=(h_1-I)x=f(x,u)u
\end{equation}
Hence from (4.7) of \cite{bh} we have that for their orthogonal basis $u_j$ of $\mathbb R^n$, $j=1,2,\dots,n$,
\begin{equation}\label{eq2}
f(u_j,u_j)=i\exp\left[2\pi i\left(\frac{\beta_1+\beta_2+\beta_n-\alpha_1-\alpha_2-\alpha_n}{2}\right)\right]\cdot 2\sin\pi(\beta_j-\alpha_j)\prod_{k\not=j}\frac{\sin\pi(\beta_k-\alpha_j)}{\sin\pi(\alpha_k-\alpha_j)}
\end{equation}
where we have assumed for simplicity that the $\alpha_k$'s and $\beta_k$'s are distinct.

To determine the signs there are two cases:

\begin{itemize}
\item[(1)] $\alpha_1=0$:

\noindent We write $(\alpha_1,\alpha_2,\dots,\alpha_n)$ as $(0,t_1,\dots,t_m,t_{-m},\dots,t_{-1})$ with $2m+1=n$.

\begin{center}
\includegraphics[width=80mm]{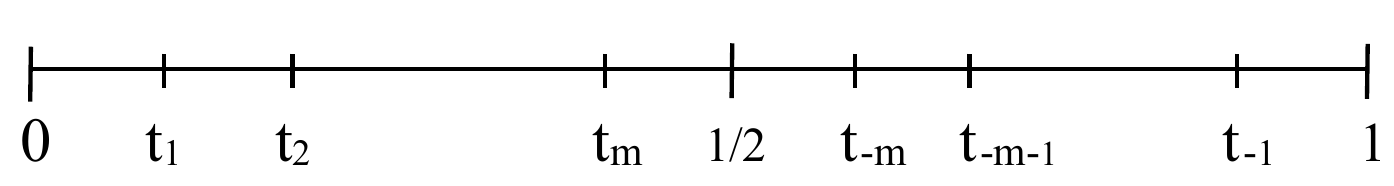}
\end{center}

\noindent Here $t_j=1-t_{-j}$ for $j=1,\dots,m$ (which corresponds to self duality of $\alpha,\beta$).

Writing $\beta$'s similarly as $(\beta_1,\dots,\beta_n)=(s_1,s_2,\dots,s_m,\frac{1}{2},s_{-m},\dots,s_{-1})$ with $s_j=1-s_{-j}$ for $j=1,\dots,m$, $2m+1=n$, we find that for the $\alpha$'s and $\beta$'s to almost interlace the configuration must be

\begin{center}
\includegraphics[width=80mm]{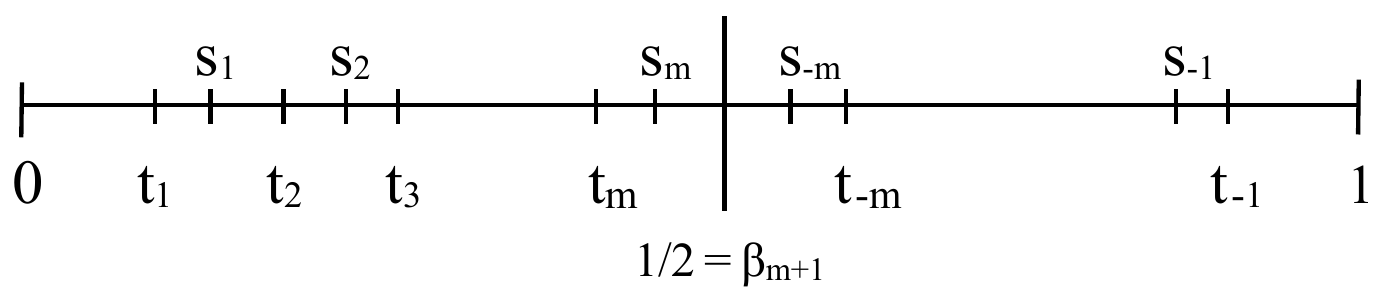}
\end{center}

\noindent The factor $i\exp[\;]$ in (\ref{eq2}) is therefore $i\exp[2\pi i/4]=-1$ and hence $f(u_1,u_1)< 0$.  Between every pair $\alpha_j,\alpha_{j+1}$ but the first one, there are one or three $\beta$'s and hence $f(u_2,u_2)>0, f(u_3,u_3)>0,\dots,f(u_n,u_n)>0$.  That is $f$ has signature $(1,1,\dots,1,-1)$.

\item[(2)] $\beta_1=0$:

\noindent Again we can write $\alpha=(\alpha_1,\dots,\alpha_n)=(t_1,\dots,t_m,\frac{1}{2},t_{-m},\dots,t_{-1})$ and $\beta=(\beta_1,\dots,\beta_n)=(0,s_1,s_2,\dots,s_m,s_{-m},s_{-m-1},\dots,s_{-1})$, with $t_{-m}=1-t_m$ and $s_{-m}=1-s_m$.

To almost interlace we must have the configuration

\begin{center}
\includegraphics[width=82mm]{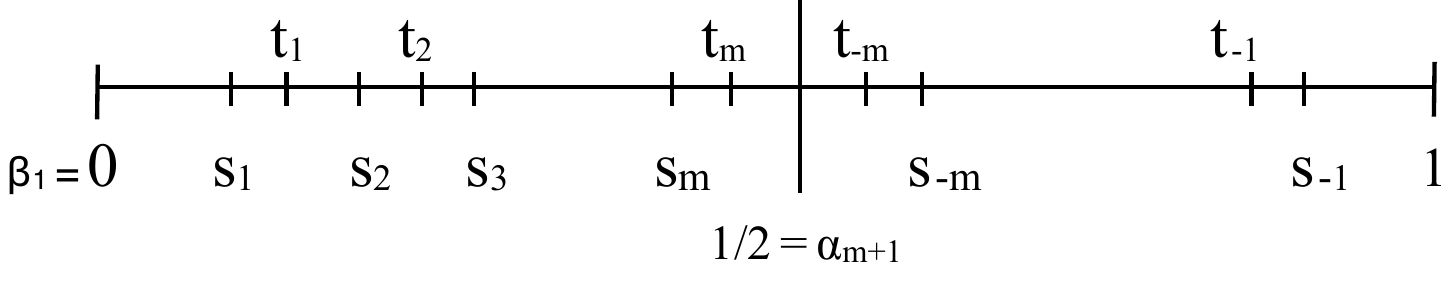}
\end{center}

\noindent Hence the factor $i\exp[\;]$ in (\ref{eq2} is this time $i\exp[-2\pi i/4]=1$.  Hence $f(u_1,u_1)>0$, as are $f(u_2,u_2),\dots, f(u_m,u_m)$.  Since there are no $\beta$'s in $(t_m,1/2)$ and $(1/2, t_{-m})$, we have that $f(u_m,u_m)<0$ and $f(u_{m+1},u_{m+1})>0$ as are the rest of the $f(u_j,u_j)$ for $m+1<j\leq n$.  Hence the signature of $f$ is again $(1,1,\dots,-1)$.
\end{itemize}

Thus in all cases (that is the two above as well as those where some of the $\alpha$'s coincide in which case one deduces the signs by a limiting argument) we have that if $f$ is normalized so that $f(v,v)=-2$ then $f$ has signature $(1,1,\dots,1,-1)$.  In particular $v$ always lies inside the null cone and $r_v$ induces a Cartan involution.  This is in sharp contrast to the involutions in $\textrm{O}_f(\mathbb Z)$ that arise in the theory of $K3$-surfaces in \cite{nik} and \cite{vin} which in our normalization have $f(v,v)=2$ and hence induce hyperbolic reflections.

\subsection{An example of a non-thin Cartan subgroup}
The goal of this short section is to give an example of a hyperbolic lattice $\textrm{O}(L)$ whose reflection group $R_2(L)$ is thin while the Cartan subgroup $R_{-2}(L)$ is of finite index.

Define a quadratic form 
\begin{equation}
f(x_1,x_2,x_3,x_4):=2x_1^2+5x_2^2+10x_3^2-x_4^2
\end{equation}
and let $\textrm{O}(L)$ be the orthogonal group of $f$ over $\Z$.
The following facts can be found 
in pages 474--477 of \cite{EGM}.
\begin{itemize}
\item[Fact 1] $R_2(L)$ is thin.
\item[Fact 2] $R_2(L)$ is generated by 4 reflections
$\sigma_1,\sigma_2,\sigma_3,\sigma_4$ and $\sigma_2,\sigma_3$
belong to the center of $R_2(L)$. Thus, $R_2(L)$ is isomorphic to a quotient of $C_2 \times C_2 \times \Di_\infty$ where $\Di_\infty$ is the infinite dihedral group.
In particular any element of $R_2(L)$ which has infinite 
order generates a finite index subgroup. 
\item[Fact 3] $\textrm{O}(L)$ is isomorphic to the semidirect product 
$R_2(L) \mathbb{o} \Di_\infty$. Under this isomorphism 
$\Di_\infty$ is generated by the following two involutions:
$$
g:=\left(\begin{array}{cccc}
1 & 0 & 0 & 0 \\
0 &\minus2 &\minus2 & 1 \\
0 &\minus1 &\minus3 & 1 \\
0 &\minus5 &\minus10& 4
\end{array}\right), \ \ \ \ \ 
h:=\left(\begin{array}{cccc}
\minus2 & 0 & \minus5 & 2 \\
0  & 1 &  0 & 0 \\
\minus1 & 0 & \minus6 & 2 \\
\minus4 & 0 & \minus20& 7
\end{array}\right).
$$
\end{itemize}

We start by showing that $R_{-2}(L)$ is not contained 
in $R_2(L)$. A direct computation shows that $\Q^4$ is decomposed as $V^+ \oplus V^-$ where 
$V^+:=\{v \in \Q^4\mid hv=v\}$ and $V^-:=\{v \in \Q^4\mid hv=-v\}$. In addition $u:=\left(\begin{array}{cccc}
1 & 1 & 1 & 4 \end{array}\right)^t\in V^+$ satisfies 
$f(u)=1$ so the reflection $r_u$ in integral.
Let $w$ be an element of $V^+$ orthogonal to $u$. Since 
$f|_{V^+}$ is not positive definite, $(w,w)<0$. Notice that:
\begin{equation}
\forall v \in V^-.\ hr_u(v)=-v,
\end{equation}
\begin{equation}
hr_u(u)=-u,
\end{equation}
\begin{equation}
hr_u(w)=w.
\end{equation}
Thus, the Cartan involution $r_w=hr_u$ is integral and does not belong to $R_2(L)$. 

Denote $r_1:=r_w$, $r_2:=r_w^{gh}$ and $\Lambda:=\langle r_1,r_2\rangle$. The Cartan involutions $r_1$ and $r_2$
have distinct images in $\textrm{O}(L)/R_2(L) \simeq \Di_\infty$. Thus, $\Lambda R_2(L)/R_2(L)$ is of finite index in $\textrm{O}(L)/R_2(L)$. In other words $R_2(L)\Lambda$ is of finite index in $\textrm{O}(L)$. 

By the second isomorphism theorem,
\begin{equation}
\Lambda / (\Lambda \cap R_2(L)) \simeq \Lambda R_2(L)/R_2(L).
\end{equation}
Every proper quotient of $\Di_\infty$ is finite so 
$\Lambda \cap R_2(L)$ is trivial. In particular, 
the group generated by $\Lambda$ and $R_2(L)$ is in fact a semidirect product. Assume for the moment that $\sigma_1$
not commute with $r_1$. Then $r_3:=\sigma_1 r_1 \sigma_1 \in R_2(L) \Lambda$
is a Cartan involution different form $r_1$ so 
$r_1r_3$ has infinite order. Moreover, since $R_2(L)$ is normal in $\textrm{O}(L)$, $r_1r_3 \in R_2(L)$ 
so by Fact 2 the group generated by $r_1r_3$ has finite index in $R_2(L)$. Altogether we get that $\langle r_1,r_2,r_3 \rangle$ has finite index in $R_2(L)\Lambda$ and in $\textrm{O}(L)$. 

Finally, $\sigma_1$ and $r_1$ do not commute since

$$
r_1:=\left(\begin{array}{cccc}
\minus6 & \minus10 & \minus25 & 10 \\
\minus4 &\minus9 &\minus20 & 8 \\
\minus5 &\minus10 &\minus26 & 10 \\
\minus20 &\minus40 &\minus100& 39
\end{array}\right), \ \ \ \ \ 
\sigma_1:=\left(\begin{array}{cccc}
\minus1 & 0 & 0 & 0 \\
0  & 1 &  0 & 0 \\
0 & 0 & 1 & 0 \\
0 & 0 & 0& 1
\end{array}\right).
$$

\section{The minimal distance graph and thin families}\label{mindistgraph}
\subsection{The distance graph}\label{graphprop}

In this section we derive a sufficient condition for a subgroup $\Delta$ of $\textrm{O}(L)$ which is generated by Cartan involutions, to have finite image in $\textrm{O}(L)/R_k(L)$ with $k=2$ or $4$.  Here $L$ is an integral quadratic lattice of signature $(-1,1,\dots,1)$ and our notation is the same as in Section~\ref{cartaninvol}.  We say $L$ is \emph{even} if $(x,x)=f(x)$ is even for all $x\in L$ but $f(x,y)$ is odd for some $x,y\in L$.  If the latter fails, that is $f(x,y)/2$ is integral for all $x,y\in L$ and $f(x,x)/2$ is odd for some $x\in L$ (which will always be the case for us) then $f/2$ is odd and if there is no room for confusion we say that $f$ itself is \emph{odd}.  If $f$ is even we will make use of the hyperbolic reflection group $R_2(L)$, while if $f$ is odd we will use $R_4(L)$.  The following diophantine lemma is crucial for what follows.

\begin{lemma}\label{basic} If $u,w \in \Z^n$ satisfy $(u,u)=(w,w)=-2$ and $(u,w)=-3$ then 
\begin{itemize}
\item[(i)] If $f$ is even and $u$ and $w$ are in $V_{-2}(L)$ with $(u,w)=-3$ then
\begin{itemize}
\item[(1)] $(u-w,u-w)=(u-2w,u-2w)=2$.
\item[(2)] $r_ur_w=r_{u-w}r_{u-2w}$.
\item[(3)] $r_{u-w}(u)=w$.
\end{itemize}
\item[(ii)] If $f$ is odd (i.e. $(\;,\;)/2$ is integral and odd) and $u$ and $w$ are in $V_{-2}(L)$ with $(u,w)=-4$ then 
\begin{itemize}
\item[(1)] $(u-w,u-w)=(u-3w,u-3w)=4$.
\item[(2)] $r_ur_w=r_{u-w}r_{u-3w}$.
\item[(3)] $r_{u-w}(u)=w$.
\end{itemize}
\end{itemize}
\end{lemma}
\begin{proof}

Part (1) of both (i) and (ii) is immediate.  As for parts (2) and (3) of (i), note that since $r_u,r_w,r_{u-w},$ and $r_{u-2w}$ all fix the orthogonal complement $\langle u,w\rangle^\perp$ of the span of $u$ and $w$, it suffices to check (2) on the two dimensional space $\langle u,w \rangle$.  That is to check the identity on $u$ and $w$.  Now a direct calculation shows that 
\begin{align*}
&r_u(u)=-u, \;\; r_u(w)=w-3u\\
&r_w(u)=u-3w, \;\; r_w(w)=-w\\
&r_{u-w}(u)=w, \;\; r_{u-w}(w)=u
\end{align*}
and
\begin{align*}
&r_{u-2w}(u)=-3u+8w, \;\; r_{u-2w}(w)=-u+3w
\end{align*}
Hence with respect to the basis $u,w$ of $\langle u,w\rangle$ we have
\begin{equation*}
r_u=\left[\begin{array}{rr}
\minus 1& \minus 3\\
0& 1
\end{array}\right], \;\;
r_w=\left[\begin{array}{rr}
1&0\\
\minus 3& \minus 1
\end{array}\right], \;\;
r_{u-w}=\left[\begin{array}{rr}
0&1\\
1& 0
\end{array}\right], \;\;
r_{u-2w}=\left[\begin{array}{rr}
\minus 3& \minus 1\\
8& 3
\end{array}\right]
\end{equation*}
and parts (2) and (3) of (i) follow.  Similarly for part (2) and (3) of (ii).  This time in the basis $u,w$ of $\langle u,w\rangle$ we have that
\begin{equation*}
r_u=\left[\begin{array}{rr}
\minus 1& \minus 4\\
0& 1
\end{array}\right], \;\;
r_w=\left[\begin{array}{rr}
1&0\\
\minus 4& \minus 1
\end{array}\right], \;\;
r_{u-w}=\left[\begin{array}{rr}
0&1\\
1& 0
\end{array}\right], \;\;
r_{u-3w}=\left[\begin{array}{rr}
\minus 4& \minus 1\\
15& 4
\end{array}\right]
\end{equation*}
and hence (2) and (3) of (ii) hold.
\end{proof}
\begin{remark}
The factorization of $r_ur_w$ as a product of two \emph{integral} hyperbolic reflections in part (2) of both (i) and (ii) in the previous lemma has its source in the binary integral quadratic forms $f_1=-2x^2-6xy-2y^2$ and $f_2=-2x^2-8xy-2y^2$ being reciprocal in that $f_1$ and $f_2$ are integrally equivalent to $-f_1$ and $-f_2$, respectively.  As shown in \cite{sa} this property is quite rare and one can check directly (or by the classification in \cite{sa}) that $x^2+kxy+y^2$ where $k\geq 3$ is reciprocal if and only if $k=3$ or $4$ and these two cases correspond to $f_1$ and $f_2$ above.
\end{remark}
Lemma~\ref{basic} leads us to the definition of the minimal distance graph associated with $f$.
\begin{dfn} The minimal distance graph $X_f$ of $f$ is the graph with vertex set $V_{-2}(L)$ (i.e. the set of Cartan involutions coming from length $-2$ roots) and edge sets
$$E_f:=\{\{u,w\}\mid (u,w)=-3\} \text{ if } f \text{ is even}$$ and
$$E_f:=\{\{u,w\}\mid (u,w)=-4\} \text{ if } f \text{ is  odd}.$$
\end{dfn}
The name minimal distance graph comes from the fact that if $u$ and $w$ are in $V_{-2}(L)$ then as points of hyperbolic space $\mathbb H^{n-1}\cong V_{-2}(L)\otimes \mathbb R$ the distance from $u$ to $w$ is equal to $\cosh^{-1}\left(\frac{-(u,w)}{2}\right)$.  Hence if $(u,w)=-3$ when $f$ is even or $-4$ when $f$ is odd, then clearly the integrality of $(\;,\,)$ on $L$ implies that $(u,w)$ has the minimal admissible distance as members of $V_{-2}(L)$.

\begin{prop}\label{lemma min dist graph} Let $S$ be a connected component of $X_f$, then the image of $\langle r_u \mid u \in S\rangle$ in 
$\textrm{O}(L)/R_2(L)$ (respectively $\textrm{O}(L)/R_4(L)$) is of order at most two.
\end{prop}
\begin{proof} Let $w_1,\ldots,w_m$ be a path in $X_f$. We have
$$r_{w_1}r_{w_m}=(r_{w_1}r_{w_2})(r_{w_2}r_{w_3})\cdots(r_{w_{m-1}}r_{w_m}),$$
which according to Lemma \ref{basic} in the case that $f$ is even (and a similar conclusion when $f$ is odd) gives
$$r_{w_1}r_{w_m}=(r_{w_1-w_2}r_{w_1-2w_2})(r_{w_2-w_3}r_{w_2-2w_3})\cdots (r_{w_{m-1}-w_m}r_{w_{m-1}-2w_m}).$$
Now the elements in each parenthesis above lie in $R_2(L)$, and hence $r_{w_1}r_{w_m}\in R_2(L)$ from which the proposition follows directly.
\end{proof}

Our certificate for a subgroup $\Delta$ of $\textrm{O}(L)$ generated by Cartan roots $v$ in $V_{-2}(L)$ to be infinite index in $\textrm{O}(L)$ is now clear.  If the generators of $\Delta$ all lie in the same connected component of $X_f$ then the image of $\Delta$ in $\textrm{O}(L)/R_k(L)$ where $k=2$ or $4$ is finite.  Hence if $|\textrm{O}(L)/R_k(L)|=\infty$ then $\Delta$ is thin.

We turn to the structure of the minimal distance graph $X_f$ which can be determined both theoretically and algorithmically. In particular we show that the question whether $u$ and $z$ belong to the same connected component of the minimal distance graph can be decided effectively. 

First note that $O(L)$ acts on $X_f$ isometrically (that is every 
$\gamma \in O(L)$ maps $X_f$ to $X_f$ as a graph isomorphism) and with finitely many orbits. These actions permutes the connected components of $X_f$ and thus there are only finitely many types of components. To examine the the components we identify a subgroup 
of $O(L)$ which stabilizes and acts transitively on a given component $\Sigma$.

\begin{prop}\label{prop fund dom of ref gp} Let $\Sigma$ be a connected component of $X_f$ and let $u$ be a vertex of $\Sigma$. Assume that $u$ is not the only vertex in $\Sigma$ and denote by $u_1,\ldots,u_k$ its immediate neighbors. Let 
\begin{equation}
G_u:=\langle r_{w_j}\mid w_j=u-u_j,j=1,2,\ldots,k\rangle.
\end{equation}
Then $G_u$ preserves $\Sigma$ and acts transitively. 
\end{prop}
\begin{proof} Lemma~\ref{basic} shows that $r_{w_j}\in O(L)$ and since 
$r_{w_j}$ switches $u$ and $u_j$ it follows that $r_{w_j}(\Sigma)=\Sigma$ and hence $G_u(\Sigma)=\Sigma$. To see that it acts transitively we proceed by induction on the distance of points in $\Sigma$ from $u$ (in the graph metric). The $u_j$'s with $d_X(u,u_j)=1$ are joined to $u$. By the choice of the $r_{w_j}$'s these points are in $G_u(u)$. For $u^* \in \Sigma$ and $d_X(u,u^*)=m>1$ pick $\bar{u} \in \Sigma$ for which $d_X(u,\bar{u})=m-1$ and $d_X(\bar{u},u^*)=1$. By induction there is a $g \in G_u$ such that
$g(\bar{u})=u$ and hence $g(u^*)=u_j$ for some $j$. Thus, $u^*=g^{-1}r_{w_j}(u)$ and $u^* \in G_u(u)$.
\end{proof}

We continue our examination of $\Sigma$ and $G_u$ with a general discussion of finding algorithmically a fundamental domain for any discrete group of motions generated by a finite number of reflections $t_1,\ldots,t_k$ (below we will apply this to $G_u$).
Now, given any discrete group $\Phi$ of hyperbolic motions of $\mathbb{H}^n$ which is generated by reflections it has a polyhedral fundamental domain as follows: Let $R$ be the set of \emph{all} reflections in $\Phi$. For $r \in R$ let $H_r$ denote the corresponding hyperplane fixed by $r$ (the ``wall'' of r). The set of $H_r$'s decompose $\mathbb{H}^n$ into connected convex polyhedral 
cells  each of which is a fundamental domain for $\Phi$. These cells are defined canonically and up to their motions by $\Phi$ are unique. We are interested in determining one of these cells algorithmically. Assuming that one can determine all the $H_r$'s which meet any compact set  (there being finitely many by discreteness), Vinberg \cite{vin} gives and effective algorithm to compute a cell of $\Phi$. In our case of $\Phi$ being generated by $t_1,\ldots,t_k$ we don't know apriori which $r$'s are in $R$ so we proceed with a somewhat different algorithm.

For a suitably chosen point $v \in \mathbb{H}^n$ the algorithm computes the cell $P_v$ of $\Phi$ which contains $v$ as an interior point. Here suitable means that $v$ does not lie in any of the $H_r's$ for $r \in R$. While such $v$ exist by discreteness of $\Phi$, choosing $v$ might look problematic since we don't have a list of the members 
of $R$. However all that we need is an upper bound on $R$ and this is something that usually comes with the source of $\Phi$ being discrete. So if $\Phi \subseteq O(L)$ then $R$ is contained in the set of all reflections in $O(L)$ which locally is easy to determine.
These provide us with effective upper bounds for $R$ and an ample supply of explicit points for $v$.  Given $\Phi \subseteq O(L)$ let
$\bar{R}$ be the set of reflections in $O(L)$ and define 
$Q:=\{H_r \mid r \in \bar{R}\}$. For $H \in Q$ let $d(v,H)$ be the distance in $\mathbb{H}^n$ form $v$ to $H$. There is a unique $w_H$ in $H$ which is the point in $H$ closest to $v_j$, so that 
$d(v,H)=d(v,w_H)$. The set $D:=\{d(v,H)\mid H \in Q\}$ is a discrete subset of $(0,\infty)$ and has a smallest element. Note that the size of the smallest element can be effectively computed from the form preserved by $O(L)$. For a finite multiset $T \subseteq Q$ (we allow repetitions in $T$) define the height $h(T)=h_v(T)$ to be 
\begin{equation}
h(T):=\sum_{H \in T}d(v,H)=\sum_{h \in T}d(v,w_H).
\end{equation}
It is clear from the properties of $D$ that 
$$\{h(T) \mid T \subseteq Q\ \wedge |T|<\infty\}\cap[0,x]$$ is finite for every $x>0$ and can be effectively bounded from above in terms of $x$. 

Let $T=\{H_1,\ldots,H_l\}\subseteq R$ be a multiset and let $G(T)$ denote the group generated by 
$t_1,\ldots,t_l$ where $t_j$ is the reflection in $H_j$. The algorithm inputs $R:=\{H_1,\ldots,H_l\}$ corresponding to the generators $r_1,\ldots,r_l$ of $\Phi$ and proceeds to replace 
$T$ with a multiset $\bar{T} \subseteq R$ by one of 3-steps, each of which reduces the height of $T$ and possibly $|T|$ while maintaining 
$G(T)=G(\bar{T})$. This process is repeated until no moves are possible and this happens after effectively bounded number of steps. 
When this process stops it outputs the fundamental cell for $\Phi$ containing $v$ as an interior point. \\ \\
{\bf Step 1 (Repetition):} If any of the $H_j$'s in $T$ are duplicated then remove one of them to $\bar{T}$. Clearly $h(\bar{T})<h(T)$, $|\bar{T}| < |T|$ and $G(T)=G(\bar{T})$.
\\ \\
{\bf Step 2 (Redundant faces):}
For every one of the hyperplanes $H_1,\ldots,H_l$ in $T$ let $H_j^-$ be the closed half space of $\mathbb{H}^n$ determined by $H_j$ and 
containing $v$. Let
\begin{equation}
K:=\cap_{1 \le j \le l} H_j^-
\end{equation}
which clearly contains $v$ as an interior point. Assume that some
$H_j$, say $H_1$, is redundant (this means $K=\cap_{2 \le j\le k}H_j^-$). There are two cases to consider. 

If $w_{H_1} \not \in P$ then there must be some $H_j$, call it $H_2$, which separates $w_{H_1}$ from $v$.  
But then 
$\dist(v,t_2(H_1))\le \dist(v,t_2(w_{H_1})) <\dist(v,w_{H_1})$. Hence, if we replace $H_1$ by $t_2(H_1)$ and $t_1$ by $t_2t_1t_2$
we get $\bar{T}:=\{t_2(H_1),H_2,\ldots,H_l\}$ with $|\bar{T}|=|T|$, $h(\bar{T})<h(T)$ and $G(\bar{T})=G(T)$.

If $w_{H_1} \in K$ then there must be some $H_j$, call it $H_2$, such that $H_2$ passes through $w_H$.
Let $M$ be the 2-dimensional 
plane containing $v$, $w_{H_1}$ and $w_{H_2}$. Note that $M$ is invariant under $t_1$ and $t_2$. Since
$H_2$ does not contain $v$, $H_2 \cap M$ does not contain 
the geodesic connecting $v$ to $w_{H_1}$. Thus, $H_2 \cap M$ is not orthogonal to $H_1 \cap M $ so $\dist(v,t_2(H_1))<\dist(v,H_1)$. Hence, if we replace $H_1$ by $t_2(H_1)$ and $t_1$ by $t_2t_1t_2$
we get $\bar{T}:=\{t_2(H_1),H_2,\ldots,H_l\}$ with $|\bar{T}|=|T|$, $h(\bar{T})<h(T)$ and $G(\bar{T})=G(T)$.
\\ \\
{\bf Step 3 (Dihedredal angels):}
Let $T$ and $K$ be as above. Suppose that two $(n-1)$-dimensional faces, say $H_1$ and $H_2$, meet and the dihedral angle $\alpha$ between them is not a submultiple of $\pi$. We can furthermore assume that $\dist(v,H_2) \le \dist(v,H_1)$.
Since $\Phi$ is discrete and $\alpha$ is not a submultiple of $\pi$, $\alpha=\frac{p}{q}\pi$ with $(p,q)=1$ and $p \ge 2$ .
Let $M$ be the 2-dimensional plane containing $v$, $w_{H_1}$,and $w_{H_2}$. One checks that
$M$ is invariant under $t_1$ and $t_2$ and the intersection 
$M \cap H \cap L$ is a point $p$. Moreover, the dihedral angle 
between the geodesics $h_1:=H_1 \cap M$ and $h_2:=H_2 \cap M$ is $\alpha$. Thus, we are reduced to the case of the hyperbolic plane
with reflections in $h_1$ and $h_2$ which meet at an angle $\alpha$.
We identify $t_1$ and $t_2$ with their restrictions to $M$
(both act as the identity on the $(n-2)$ dimensional plane passing through $p$ and orthogonal to $M$). Now, $g:=t_1t_2$ is a rotation
about $p$ by an angle of $2\alpha=\frac{2p}{q}\pi$. Thus, there exists an $m$  such that $\bar{g}=g^m$ is a rotation about $p$ with by an angle of $\frac{2}{q}\pi$. There are two cases to consider.

If $p$ is odd choose $s$ such that $\bar{h}:=\bar{g}^s(h_1)\cap M$ 
intersects the interior of $M \cap h_1^- \cap h_2^-$
and the angle between $\bar{h}$ and $h_2$ is $\frac{1}{q}\pi$
(where $h_i^-:=M \cap H_i^-$). 
Then $\dist(v,\bar{h})<\dist (v,h_1)$. Hence, if we replace $H_1$ by $\bar{g}^s(H_1)$ and $t_1$ by $\bar{g}^st_1\bar{g}^{-s}$
we get $\bar{T}:=\{\bar{g}^s(H_1),H_2,\ldots,H_l\}$ with $|\bar{T}|=|T|$, $h(\bar{T})<h(T)$ and $G(\bar{T})=G(T)$ (the last assertion is true since that group generated by the reflections in $h_1$ and $\bar{h}$ contains a rotation about $p$ by angle $\frac{2p}{q}\pi$ so it must be equal to the group generated by $t_1$ and $t_2$).

If $p$ is even then $q$ must be odd. Choose $s$ such that $\bar{h}:=\bar{g}^s(h_2)\cap M$ 
intersects the interior of $M \cap h_1^- \cap h_2^-$
and the angle between $\bar{h}$ and $h_2$ is $\frac{1}{q}\pi$. Then $\dist(v,\bar{H})<\dist (v,h_1)$. Hence, if we replace $H_1$ by $\bar{g}^s(H_2)$ and $t_1$ by $\bar{g}^st_2\bar{g}^{-s}$
we get $\bar{T}:=\{\bar{g}^s(H_2),H_2,\ldots,H_l\}$ with $|\bar{T}|=|T|$, $h(\bar{T})<h(T)$ and $G(\bar{T})=G(T)$.
\\ 

After repeating the algorithm a finite number of times, we arrive at a reduced $T=\{H_1,\ldots,H_l\}$ , namely one for which we cannot apply any of the above steps and we set $P_v:=\cap_{i=1}^l H_{t_j}^-$. The polyhedron $P_v$  contains $v$ and an interior point and since we can't apply step 2, $P_v \cap H_j$ is an $(n-1)$-dimensional face of $P_v$ for every $1 \le j \le l$ (and any $(n-1)$-dimensional face is of this form). Moreover, all the dihedral angles between pairs of the $(n-1)$-faces of $P_v$ are submultiplies 
of $\pi$ (including 0 as a submultiple when the faces aren't adjacent and the corresponding $H_j$'s don't intersect in $\mathbb{H}^n$).
It follows from Vinberg \cite{vin} that $P_v$ is a fundamental polyhedron 
for the group generated by reflections in the $(n-1)$-dimensional
faces of $P_v$ which is just the group $\Phi$. Moreover, this group is a Coxeter group with generators $t_1,\ldots,t_l$ and the usual relations are derived  from the corresponding Coxter matrix. This completes the general discussion of finding a fundamental domain for $\Phi$.

Returning to the group $G_u$ defined in Proposition \ref{prop fund dom of ref gp} which is generated by the reflections $r_1,\ldots,r_k$. Applying the general algorithm we find that $G_u=\langle t_1,\ldots, t_l\rangle$
with $l \le k$ and $t_1,\ldots,t_l$ reflections in the faces of a fundamental reflective polyhedron $P_v$. Using these canonical generators of $G_u$ we can decide if a given $z \in X_f$ is in the same connected component as $u$, that is whether $z \in G_u(u)$.
With $V_{-2}$ as a model for $\mathbb{H}^n$ it is a matter of reducing $z$ into $P_v$ and checking if this reduction is the same as the one for $u$ into $P_v$ (every orbit of $G_u$ has a unique representative on $P_v$, this is true also for orbits intersecting  the boundary of $P_v$). Now, if $z \not \in P_v$ then there is some $t_j$ such that $z \not \in H^-_{t_j}$, but then 
$\dist(v,t_j(z))<\dist(v,z)$. Of course $t_j(z)$ is the same $G_u$ orbit as $z$. Repeat this reduction to $t_j(z)$ if it is not in $P_v$ and continuing this for a finite number of steps, must place it in $P_v$ (by discreteness). If the point we arrive at is  the same as the reduction of $u$ into $P_v$ then $z \in \Sigma=G_u(u)$, otherwise 
$z \not \in \Sigma$.

\subsection{Thin Families}\label{thinguys}
Using the minimal distance graph and the results reviewed in Section~\ref{cartaninvol}, we apply our thinness certificate to various hyperbolic hypergeometric monodromy groups.  Recall that every $H=H(\alpha,\beta)$ is generated by two 
matrices $A,B \in \GL_{n}(\Z)$ for some odd $n\in \N$.  Moreover, there is an integer quadratic form $(u_1,u_2)={u_1}^tfu_2$ of signature $(n-1,1)$ preserved by $H$. The matrix $C:=A^{-1}B$ has an eigenvalue $-1$ of algebraic multiplicity $1$ with a corresponding eigenvector $v\in \Z^n$. The elements $v,Bv,\ldots,B^{n-1}v$ are linearly independent and the $\Z$-lattice $L$ they span is preserved by $H$. Proposition~\ref{prop form} implies we can normalize $f$ to have $(v,v)=-2$ and still be integral on $L$. The group $H$ is of infinite index in the integral orthogonal group of the original form if and only if its restriction to $L$ is of infinite index in the integral orthogonal group of the normalized form. Thus by restricting to $L$, we can assume from now that $f$ is normalized so $(v,v)=-2$ and $-C$ is a Cartan involution.   
If $B$ has finite order then the group 
$\langle B^{i}CB^{-i}\mid i \in \N\rangle$ is a finite index normal subgroup of $H$. Thus, $H$ is thin if and only if 
$$H_r:=\langle B^{i}(-C)B^{-i}\mid i \in \Z\rangle$$ is thin in $\Or_f(Z)$.
The advantage in considering $H_r$ is clear, this group is generated by Cartan involutions so we can apply Proposition~\ref{lemma min dist graph}.  We begin by noting the following.

\begin{lemma}\label{lemma v-Bv} If $v$ and $Bv$ belong to same component of $X_f$ then $H$ is thin. 
\end{lemma}
\begin{proof} The group $H_r$ is generated by the set of Cartan involutions $\{r_{B^iv}\mid i \in \Z\}$. Since $B$ preserves $f$, if 
$v$ and $Bv$ belong to the same connected component then also $B^iv$ and $B^{i+1}v$ belong to the same connected component for every $i \in \N$. Thus, $\{r_{B^iv}\mid i \in \Z\}$ is contained in one connected component so $H_r$ is thin by Proposition~\ref{lemma min dist graph}. Hence, $H$ is thin by the above paragraph. 
\end{proof}

We now prove the following.

\begin{thm}\label{thinfamilythm} The following subfamilies of hyperbolic hypergeometric monodromy groups are thin ($n$ is always odd):
\begin{itemize}
\item[(1)] $\mathcal N_1(1,n,n)$, $n>3$
$$\hspace{-1.7in} {\scriptstyle \alpha=(0,\frac{1}{n+1},\dots,\frac{n-1}{2(n+1)},\frac{n+3}{2(n+1)},\dots,\frac{n}{n+1}), \; \beta=(\frac{1}{2},\frac{1}{n},\frac{2}{n},\dots,\frac{n-1}{n})}$$
\item[(2)] $\mathcal M_1(1,n)$, $n>3$
$$\hspace{-.5in} {\scriptstyle \alpha=(0, \frac{1}{2n}, \frac{3}{2n}, \dots, \frac{n-2}{2n}, \frac{n+2}{2n},\dots, \frac{2n-3}{2n}, \frac{2n-1}{2n}), \; \beta=(\frac{1}{2}, \frac{1}{2n-2}, \frac{3}{2n-2}, \dots, \frac{2n-5}{2n-2}, \frac{2n-3}{2n-2})}$$
\item[(3)] $\mathcal N_1(1,1,n)$, $n>30$
$$\hspace{-1.3in} {\scriptstyle \alpha=(0, \frac{1}{2n}, \frac{3}{2n}, \dots, \frac{n-2}{2n}, \frac{n+2}{2n},\dots, \frac{2n-3}{2n}, \frac{2n-1}{2n}), \; \beta=(\frac{1}{2},\frac{1}{n},\frac{2}{n},\dots, \frac{n-1}{n})}$$
\item[(4)] $\mathcal M_2(n-2,n)$, $n>3$
$$\hspace{-1.3in} {\scriptstyle \alpha=(\frac{1}{2}, \frac{1}{2n-2}, \frac{3}{2n-2}, \dots, \frac{2n-5}{2n-2}, \frac{2n-3}{2n-2}), \; \beta=(0,0,0,\frac{1}{n-2},\frac{2}{n-2},\dots, \frac{n-3}{n-2})}$$
\item[(5)] $\mathcal M_2((n-1)/2,n)$, $n>30$, $n\not\equiv 1$ mod $4$
$${\scriptstyle \alpha=(\frac{1}{2}, \frac{1}{2n-2}, \frac{3}{2n-2}, \dots, \frac{2n-5}{2n-2}, \frac{2n-3}{2n-2}), \; \beta=(0,0,0,\frac{1}{n-1},\frac{2}{n-1},\dots, \frac{n-3}{2(n-1)},\frac{n+1}{2(n-1)},\dots,\frac{n-2}{n-1})}$$
\item[(6)] $\mathcal N_2(1,1,n)$, $n>30$
$$\hspace{-.4in} {\scriptstyle \alpha=(0, \frac{1}{2n-2}, \frac{3}{2n-2}, \dots, \frac{2n-3}{2n-2}), \; \beta=(\frac{1}{2}, \frac{1}{2},\frac{1}{2},\frac{1}{n-1},\frac{2}{n-1},\dots, \frac{n-3}{2(n-1)},\frac{n+1}{2(n-1)},\dots,\frac{n-2}{n-1})}$$
\item[(7)] $\mathcal N_2(n-1,1,n)$, $n>3$
$$\hspace{-.9in} {\scriptstyle \alpha=(0, \frac{1}{n}, \frac{2}{n}, \dots, \frac{n-1}{n}), \; \beta=(\frac{1}{2}, \frac{1}{2},\frac{1}{2},\frac{1}{n-1},\frac{2}{n-1},\dots, \frac{n-3}{2(n-1)},\frac{n+1}{2(n-1)},\dots,\frac{n-2}{n-1})}$$
\item[(8)] $\mathcal N_1(3,n,n)$, $n>3$ and $(n+1,3)=1$
$$\hspace{-1.4in} {\scriptstyle \alpha=(0,\frac{1}{n+1},\dots,\frac{n-1}{2(n+1)},\frac{n+3}{2(n+1)},\dots,\frac{n}{n+1}), \; \beta=(\frac{1}{2},\frac{1}{3},\frac{2}{3},\frac{1}{n-2},\dots,\frac{n-3}{n-2})}$$
\end{itemize}
\end{thm}

Note that in the cases above where we require $n>30$ the form fixed by the group is odd, and we appeal to Vinberg \cite{vin} to derive thinness.  In the cases where we require $n>3$, the fixed form is even and we appeal to Nikulin \cite{nik} to conclude the group is thin by checking using a computer that none of the $n<30$ cases are two elementary or belong to the lists at the end of Section~\ref{cartan1}.

\begin{proof} 

\noindent (1): 

Corollary \ref{corol form 1} shows that $(v,Bv)=-3$ and we can apply Lemma \ref{lemma v-Bv}. 

\vspace{.1in}

\noindent (2):

In this case, we have that the generators of $H(\alpha,\beta)$ are as in (\ref{agens}) with
\begin{align*}
&A_i=0\;\;\; \mbox{for } 2\leq i\leq n-2\\
&A_i=1\;\;\; \mbox{otherwise}
\end{align*}
and
\begin{align*}
&B_n=-1\\
&B_i=(-1)^i\cdot 2\;\;\; \mbox{for } i<n.
\end{align*}
Therefore the eigenvector of $C=A^{-1}B$ with eigenvalue $1$ is thus $\mathbf v=(3,2,2,2,\dots,2,\minus 1,2)^t$.  Note that the order of $B$ is $2n$.

The quadratic form $f$ fixed by $A$ and $B$ in the basis $\{\mathbf v,B\mathbf v,\dots,B^{n-1}\mathbf v\}$ has $ij$th entry  $3$ if $|i-j|=1$, and by Lemma~\ref{basic} the basis vectors are thus in one connected component of $X_f$ since $e_i^tfe_{i+1}=3$ for all $1\leq i\leq n-1$.  To see this, note that the $n$th row of $B^m$ is of the form
$$(b_1 \;\; b_2 \;\;\cdots\;\; b_n)$$
where $b_i=0$ for $i\leq n-m-1$, $b_{n-m}=1$, and $b_i=2$ for $i>n-m$.  Therefore in particular the $n$th entry of $B\mathbf v$ is $3$ as desired.

\vspace{0.1in}

\noindent (3):

In this case the generators of $H(\alpha,\beta)$ are as in (\ref{agens}) with
\begin{align*}
&A_i=2\;\;\; \mbox{for } 1\leq i\leq n-1\\
&A_i=1\;\;\; \mbox{for } i=n
\end{align*}
and
\begin{align*}
&B_n=-1\\
&B_i=(-1)^i\cdot 2\;\;\; \mbox{for } i<n.
\end{align*}
The eigenvector of $C=A^{-1}B$ with eigenvalue $1$ is thus $\mathbf v=(4,0,4,0,\dots,4,0,2)^t$.  The order of $B$ is $2n$ as in case (i).

In this case the quadratic form $f$ fixed by $A$ and $B$ in the basis of Proposition~\ref{prop form} is odd.  This follows from the fact that every entry of $\mathbf v$ is even and thus the $n$th entry of $B^m\mathbf v$ is even for every $0\leq m\leq n-1$.  Furthermore, $f_{ij}=4$ if $|i-j|=1$ since the $n$th entry of $B\mathbf v$ is $4$ so $e_i^tfe_{i+1}=4$ for $1\leq i\leq n-1$.  Combined with the fact that $f$ is odd, Lemma~\ref{basic} gives us that all of the basis vectors are in the same connected component of $X_f$.

\vspace{0.1in}

\noindent (4):

In this case we have that the generators of $H(\alpha,\beta)$ are as in (\ref{agens}) with
\begin{align*}
&A_i=(-1)^i\;\;\; \mbox{for } i=2,n-2,n\\
&A_i=(-1)^i\cdot 2\;\;\; \mbox{for } i=1,n-1\\
&A_i=0\;\;\; \mbox{otherwise.}
\end{align*}
and
\begin{align*}
&B_i=0\;\;\; \mbox{for } 2\leq i\leq n-2\\
&B_i=1\;\;\; \mbox{otherwise}
\end{align*}
The eigenvector of $C=A^{-1}B$ with eigenvalue $1$ is thus $\mathbf v=(3,\minus1,0,\dots,0,1,\minus1,2)^t$.  Here we have $B$ is of order $2n-2$.

Note that if $f$ denotes the form fixed by $A$ and $B$ in the basis of Proposition~\ref{prop form} we get that $f_{ij}$ is $3$ if $|i-j|=1$, and Lemma~\ref{basic} implies that the basis vectors arein one connected component of $X_f$ since $e_i^tfe_{i+1}=3$ for all $1\leq i\leq n-1$.  To see this, note that the $n$th row of $B^m$ is of the form
$$(b_1 \;\; b_2 \;\;\cdots\;\; b_n)$$
where $b_i=0$ for $i\leq n-m-1$ and $b_{i}=i-n+m+1$ for $i\geq n-m$.  Therefore in particular the $n$th entry of $B\mathbf v$ is $3$ as desired.

\vspace{0.1in}

\noindent (5):

In this case we have that the generators of $H(\alpha,\beta)$ are as in (\ref{agens}) with
\begin{align*}
&A_i=-1\;\;\; \mbox{for } i=n\\
&A_i=(-1)^i\cdot 3\;\;\; \mbox{for } i=n-1,1\\
&A_i=(-1)^i\cdot 4\;\;\; \mbox{for } 2\leq i\leq n-2
\end{align*}
and
\begin{align*}
&B_i=0\;\;\; \mbox{for } 2\leq i\leq n-2\\
&B_i=1\;\;\; \mbox{otherwise}
\end{align*}

The eigenvector of $C=A^{-1}B$ with eigenvalue $1$ is thus $\mathbf v=(4,\minus4,4\minus4,\dots,4,\minus2,2)^t$ and $B$ is of order $2n-2$.

As in case (4), the quadratic form $f$ fixed by $A$ and $B$ is odd.  Again, this is because every entry of $\mathbf v$ is even and thus the $n$th entry of $B^m\mathbf v$ is even for $0\leq m\leq n-1$.  Moreover, we have that $f_{ij}=4$ for $|i-j|=1$ since the $n$th entry of $B\mathbf v$ is $4$ so $e_i^tfe_{i+1}=4$ for $1\leq i\leq n-1$.  Combined with the fact that $f$ is odd, Lemma~\ref{basic} gives us that all of the basis vectors are in one connected component of $X_f$.

\vspace{0.1in}

\noindent (6):

Here the generators of $H(\alpha,\beta)$ are as in (\ref{agens}) with
\begin{align*}
&A_n=1\\
&A_i= 3\;\;\; \mbox{for } i=1,n-1\\
&A_i= 4\;\;\; \mbox{for } 2\leq i\leq n-2
\end{align*}
and
\begin{align*}
&B_i=(-1)^i\;\;\; \mbox{for } i=1,n-1,n\\
&B_i=0\;\;\; \mbox{for } 2\leq i\leq n-2
\end{align*}
The eigenvector of $C=A^{-1}B$ with eigenvalue $1$ is thus $\mathbf v=(4,4,4,\dots,4,2,2)^t$ and $B$ is of order $2n-2$.

Again, the quadratic form $f$ fixed by $A$ and $B$ in the basis of Proposition~\ref{prop form} is even since every entry of $\mathbf v$ is even.  Moreover, we have that $f_{ij}=4$ for $|i-j|=1$ since the $n$th entry of $B\mathbf v$ is $4$ so $e_i^tfe_{i+1}=4$ for $1\leq i\leq n-1$.  Combined with the fact that $f$ is odd, Lemma~\ref{basic} gives us that all of the basis vectors are in one connected component of $X_f$.

\vspace{0.1in}

\noindent (7):

Here we have that the generators of $H(\alpha,\beta)$ are as in (\ref{agens}) with
\begin{align*}
&A_n=1\\
&A_i= 3\;\;\; \mbox{for } i=1,n-1\\
&A_i= 4\;\;\; \mbox{for } 2\leq i\leq n-2
\end{align*}
and
\begin{align*}
&B_n=-1\\
&B_i=0\;\;\; \mbox{for } 1\leq i\leq n-1\\
\end{align*}
The eigenvector of $C=A^{-1}B$ with eigenvalue $1$ is thus $\mathbf v=(3,4,4,4,\dots,4,3,2)^t$ and $B$ is of order $n$.

Note that if $f$ denotes the form fixed by $A$ and $B$ in the basis of Proposition~\ref{prop form} we get that $f_{ij}$ is $3$ if $|i-j|=1$, and so Lemma~\ref{basic} implies that the basis vectors are in one connected component of $X_f$w since $e_i^tfe_{i+1}=3$ for all $1\leq i\leq n-1$.   To see this, we note that the $n$th entry of $B\mathbf v$ is $3$.

\vspace{0.1in}

\noindent (8): 

The basic idea is to use Lemma \ref{lemma v-Bv} again, however, to show that $v$ and $Bv$ belong to same connected component is more complicated then in the previous cases (one has to consider paths of longer length). Recall that for 
every $n \ge 7$ with $\gcd(n+1,6)=2$ there exists a unique monodromy group $H$ in the family $\mathcal N_1(3,n,n)$. For $1 \le i\le n$ let $v_i:=B^{i-1}v$
be the $i$-th basis element of a basis for $L$. 
If $u \in L$ then $[u]$ denotes the coordinates vector of $u$ with
respect to the base $v_1,\cdots,v_{n}$ and $[u]_i$ is the
$i$-coordinate of $[u]$. Corollary~\ref{corol form 3} states that the quadratic form is given by:

$$f_{i,j}:=(v_i,v_j)=\left\{\begin{array}{ccc}
-2 &\text{if}& |i-j|=0 \\
-4 & \text{if}&|i-j|=1 \\
-8 & \text{if}& |i-j|=2 \text { or } |i-j|=n-1\\
-11 & \text{if}& |i-j|=3 \text { or }  |i-j|=n-2\\
-12 & \text{otherwise} \\
\end{array}\right.$$

\begin{lemma}\label{norm 2 vector} Denote $m:=n-3$
if $n\equiv 1 \ (\Mod 6)$ and $m:=n-5$ if $ n\equiv 3 \ (\Mod 6)$.
Define $u \in L$ by
$$[u]_i:=\left\{\begin{array}{cc}
1 & \text{ if } i\equiv 1 \ (\Mod 6) \text{ and } i \le m \\
-2 &\text{ if } i\equiv 2 \ (\Mod 6) \text{ and } i \le m \\
2 &\text{ if } i\equiv 3 \ (\Mod 6) \text{ and } i \le m \\
-1 &\text{ if } i\equiv 4 \ (\Mod 6) \text{ and } i \le m \\
0 & otherwise
\end{array}\right.$$
 Then, $(u,u)=2$.
\end{lemma}
\begin{proof} The proof in by induction on $\left\lfloor\frac{n}{6}\right\rfloor$ (the integral part of
$\frac{n}{6}$). The induction base is $n=7$ and  $n=9$ and it is a
direct computation. Assume $n \ge 10$ and that the induction
hypothesis was proven for integer smaller then
$l:=\left\lfloor\frac{n}{6}\right\rfloor$. Define $a,b \in L$ by
$$[a]_i:=\left\{\begin{array}{cc}
1 & \text{ if } i\equiv 1 \ (\Mod 6) \text{ and } i \le m-6 \\
-2 &\text{ if } i\equiv 2 \ (\Mod 6) \text{ and } i \le m-6 \\
2 &\text{ if } i\equiv 3 \ (\Mod 6) \text{ and } i \le m-6 \\
-1 &\text{ if } i\equiv 4 \ (\Mod 6) \text{ and } i \le m-6\\
0 & otherwise
\end{array}\right.$$
and
$$[b]_i:=\left\{\begin{array}{cc}
1 & \text{ if } i=m-3 \\
-2 &\text{ if } i=m-2 \\
2 &\text{ if } i=m-1 \\
-1 &\text{ if } i=m\\
0 & otherwise
\end{array}\right..$$
Note that if $\tilde{f}$ is the form corresponding to the $n-6$
dimensional group in the family $j=3$ then the upper left $m-6
\times m-6$ of $\tilde{f}$ is the same as the upper left $m-6 \times
m-6$ of $F$. Thus, by the induction hypothesis we get that $(a,a)=2$
and that $(b,b)=(B^{-6(l-1)}b,B^{-6(l-1)}b)=2$. The equality $u=a+b$
implies
$$(u,u)=(a,a)+(b,b)+2(b,a)=4+\sum_{i=m-3}^{m}\sum_{j=1}^{m-6}[u]_i [u]_jf_{i,j}=2$$
since $f_{m-6,m-3}=f_{m-3,m-6}=-11$ while all the other $f_{i,j}$ in
the above range equal $-12$.
\end{proof}

\begin{lemma}\label{lemma j=3 1} Assume that $n\equiv 1 \ (\Mod 6)$ and let $u$ and $m$ be
as in Lemma  \ref{norm 2 vector}.  Denote $w:=u+v_{n-1}$. Then
$(w,w)=-2$ while $(w,v_{n-1})=(w,v_n)=-3$.
\end{lemma}
\begin{proof} Note that $f_{n-1,1}=-11$, $f_{n-1,m-1}=-11$ and
$f_{n-1,m}=-8$ while $f_{n-1,i}=-12$ for $2 \le i \le m-2$. Thus,
$$(u,v_{n-1})=\sum_{i=1}^m[u_i]f_{n-1,i}=-1,$$
$$(w,w)=(u,u)+(v_{n-1},v_{n-1})+2(v_{n-1},u)=2-2-2=-2$$
and
$$(w,v_{n-1})=(v_{n-1},v_{n-1})+(v_{n-1},u)=-1-2=-3.$$
Also, $f_{n,1}=-8$, $f_{n,2}=-11$ and $f_{n,m}=-11$ while
$f_{n-1,i}=-12$ for $3 \le i \le m-1$. Thus,
$$(u,v_{n})=\sum_{i=1}^m[u_i]f_{n,i}=1$$
and
$$(w,v_{n})=(v_{n},v_{n-1})+(v_{n-1},u)=-4+1=-3.$$
\end{proof}

\begin{lemma}\label{lemma j=3 2} Assume $n\equiv 5 \ (\Mod 6)$ and let $u$ and $m$ be
as in Lemma  \ref{norm 2 vector}.  Denote $w:=u+v_{n-2}$. Then
$(w,w)=-2$ while $(w,v_{n-2})=(w,v_{n-1})=-3$.
\end{lemma}
\begin{proof} Note that $f_{n-2,m}=-11$ and
while $f_{n-2,i}=-12$ for $1 \le i \le m-1$. Thus,
$$(u,v_{n-2})=\sum_{i=1}^m[u_i]f_{n-2,i}=-1,$$
$$(w,w)=(u,u)+(v_{n-2},v_{n-2})+2(v_{n-2},u)=2-2-2=-2$$
and
$$(w,v_{n-2})=(v_{n-2},v_{n-2})+(v_{n-2},u)=-1-2=-3.$$
Also, $f_{n-1,1}=-11$ while $f_{n-1,i}=-12$ for $2 \le i \le m$.
Thus,
$$(u,v_{n-1})=\sum_{i=1}^m[u_i]f_{n-1,i}=1$$
and
$$(w,v_{n-1})=(v_{n-1},v_{n-2})+(v_{n-2},u)=-4+1=-3.$$
\end{proof}
Lemma \ref{lemma j=3 1} and \ref{lemma j=3 2} show that
$v_{n-2}$ and $v_{n-1}$ are in the same connected component of $X_f$.
Thus, also $v=B^{3-n}v_{n-2}$ and $Bv=B^{3-n}v_{n-1}$ are in the same connected component and we can apply Lemma \ref{lemma v-Bv}. 

\end{proof}

By using the same kind of arguments 
one can show that more families are thin. However since the computations are very tedious, we will only state the result we were able to prove.

\begin{thm}\label{secondthin} The family $\mathcal N_1(j,n,n)$ where $n\equiv l$ (mod $2j$) is thin if one of the following conditions holds:
\begin{itemize}
\item[1.] $l=j$ or $l=j-2$.
\item[2.] $l \le \frac{j+1}{2}$ and $l+1$ divides $j-1$.
\item[3.] $l \ge \frac{2j-5}{2}$ and $2j-l-1$ divides $j-1$.
\end{itemize}
\end{thm}

It is possible that one might be able to remove the congruence conditions above.  For example, the first case which is not covered by the theorem above is $\mathcal N_1(7,17,17)$, in which there is a path of length $4$ in $X_f$ that connects $v$ and $Bv$ (and hence the group is thin).

\section{Numerical results and data}\label{numeric}

\subsection{Hyperbolic groups $H(\alpha,\beta)$, dimension $n\leq 9$}\label{sporadic}

In this section, we list all of the primitive hyperbolic hypergeometric monodromy groups, including the sporadic ones which do not fall into any of the families found in Section~\ref{familyclassification}.  We note that the sporadic groups are all in dimension $n\leq 9$ and are easily determined using the method from Section~\ref{interlace}, given the sporadic finite hypergeometric monodromy groups listed in \cite{bh} (or by listing for each $n$ all the hyperbolic hypergeometric $(\alpha,\beta)$'s and recording which lie in one of our seven families).  Our list is split into two tables: Table~\ref{table3}, which lists all even monodromy groups in dimension $n\leq 9$ as well as which family (if any) they fit into, and Table~\ref{table4},  which lists all odd monodromy groups in dimension $n\leq 9$ along with the associated family.  In Table~\ref{table3}, we also list the invariant factors of the associated quadratic lattice and specify in which cases we are able to prove that the monodromy group is thin (thinness here is always confirmed using the minimal distance graph method described above, with the longest path encountered being of length $3$).  In Table~\ref{table4} we specify when the subgroup $H_r$ of $H(\alpha,\beta)$ is contained in $R_f$ -- note that Nikulin's results from \cite{nik} do not apply to the odd case and hence we cannot immediately deduce thinness in these lower dimensional odd cases (although with extra work this can probably be done).  The groups $H(\alpha,\beta)$ where $[H(\alpha,\beta):H_r]=\infty$ (thus, those for which the minimal distance graph method does not prove thinness) are marked with a $\ast$ in the last column. 

\vspace{0.3in}

\begin{center}
{\renewcommand{\arraystretch}{2.5}
\begin{longtable}[h]{| l | l | p{2.3cm} | l | p{1cm} |}
    \hline
   $(\alpha,\beta)$&$(\alpha,\beta)$ with shift by $\nicefrac{1}{2}$&Factors&Family&Thin\\ \hline 
    \pbox{20cm}{${\scriptstyle \alpha=}(\frac{1}{2},\frac{1}{2},\frac{1}{2})$ \\ ${\scriptstyle \beta=}(0,0,0)$}&\pbox{20cm}{${\scriptstyle \alpha=}(0,0,0)$\\ ${\scriptstyle \beta=}(\frac{1}{2},\frac{1}{2},\frac{1}{2})$}&8&$\mathcal N_4(1,1,3)$&No\\ \hline
   
    \pbox{20cm}{${\scriptstyle \alpha=}(\frac{1}{3},\frac{1}{2},\frac{2}{3})$ \\ ${\scriptstyle \beta=}(0,0,0)$}&\pbox{20cm}{${\scriptstyle \alpha=}(0,\frac{1}{6},\frac{5}{6})$ \\ ${\scriptstyle \beta=}(\frac{1}{2},\frac{1}{2},\frac{1}{2})$}&24, 72&$\mathcal N_3(1,2,3)$&No \\ \hline
   
    \pbox{20cm}{${\scriptstyle \alpha=}(\frac{1}{4},\frac{1}{2},\frac{3}{4})$ \\ ${\scriptstyle \beta=}(0,0,0)$}&\pbox{20cm}{${\scriptstyle \alpha=}(0, \frac{1}{4},\frac{3}{4})$ \\ ${\scriptstyle \beta=}(\frac{1}{2},\frac{1}{2},\frac{1}{2})$}&8, 8&$\mathcal N_3(1,1,3)$&No \\ \hline
   
    \pbox{20cm}{${\scriptstyle \alpha=}(\frac{1}{6},\frac{1}{2},\frac{5}{6})$ \\ ${\scriptstyle \beta=}(0,0,0)$}&\pbox{20cm}{${\scriptstyle \alpha=}(0,\frac{1}{3},\frac{2}{3})$ \\ ${\scriptstyle \beta=}(\frac{1}{2},\frac{1}{2},\frac{1}{2})$}&8, 8&$\mathcal N_2(1,2,3)$&No \\ \hline
   
    \pbox{20cm}{${\scriptstyle \alpha=}(\frac{1}{3},\frac{1}{2},\frac{2}{3})$ \\ ${\scriptstyle \beta=}(0,\frac{1}{4},\frac{3}{4})$}& \pbox{20cm}{${\scriptstyle \alpha=}(0,\frac{1}{6},\frac{5}{6})$ \\ ${\scriptstyle \beta=}(\frac{1}{4},\frac{1}{2},\frac{3}{4})$}&12, 12&$\mathcal N_1(1,1,3)$&No\\ \hline
   
    \pbox{20cm}{${\scriptstyle \alpha=}(\frac{1}{3},\frac{1}{2},\frac{2}{3})$ \\ ${\scriptstyle \beta=}(0,\frac{1}{6},\frac{5}{6})$}&\pbox{20cm}{${\scriptstyle \alpha=}(0,\frac{1}{6},\frac{5}{6})$ \\ ${\scriptstyle \beta=}(\frac{1}{3},\frac{1}{2},\frac{2}{3})$}&3&$\mathcal N_2(1,1,3)$&No \\
\end{longtable}}
\end{center}

\vspace{1in}

\begin{center}
{\renewcommand{\arraystretch}{2.5}
\begin{longtable}[h]{| l | l | l | p{3.9cm} |}
    \hline
    $(\alpha,\beta)$&$(\alpha,\beta)$ with shift by $\nicefrac{1}{2}$&Family&$H_r$ contained in $R_f$?\\ \hline
\pbox{20cm}{${\scriptstyle \alpha=}(0,\frac{1}{10},\frac{3}{10},\frac{7}{10},\frac{9}{10})$ \\ ${\scriptstyle \beta=}(\frac{1}{5},\frac{ 2}{5},\frac{ 1}{2},\frac{ 3}{5},\frac{ 4}{5})$}&\pbox{20cm}{${\scriptstyle \alpha=}(\frac{1}{5},\frac{2}{5},\frac{1}{2},\frac{3}{5},\frac{4}{5})$ \\ ${\scriptstyle \beta=}(0, \frac{1}{10},\frac{3}{10},\frac{7}{10},\frac{9}{10})$}&$\mathcal N_1(1,1,5)$&Yes\\ \hline

\pbox{20cm}{${\scriptstyle \alpha=}(0,\frac{1}{8},\frac{ 3}{8},\frac{ 5}{8},\frac{ 7}{8})$ \\ ${\scriptstyle \beta=}(\frac{1}{4},\frac{ 1}{2},\frac{ 1}{2},\frac{ 1}{2},\frac{ 3}{4})$}&\pbox{20cm}{${\scriptstyle \alpha=}(\frac{1}{8},\frac{ 3}{8},\frac{1}{2},\frac{ 5}{8},\frac{ 7}{8})$ \\ ${\scriptstyle \beta=}(0,0,0,\frac{1}{4},\frac{ 3}{4})$}&\pbox{20cm}{$\mathcal M_2(2,5)$,\\$\mathcal N_2(1,1,5)$,\\ $\mathcal N_3(1,1,5)$}&Yes\\ \hline

\pbox{20cm}{${\scriptstyle \alpha=}(\frac{1}{6},\frac{ 1}{2},\frac{ 1}{2},\frac{ 1}{2},\frac{ 5}{6})$ \\ ${\scriptstyle \beta=}(0,0,0,\frac{1}{3},\frac{ 2}{3})$}&\pbox{20cm}{${\scriptstyle \alpha=}(0,0,0,\frac{1}{3},\frac{ 2}{3})$ \\ ${\scriptstyle \beta=}(\frac{1}{6}, \frac{ 1}{2},\frac{ 1}{2},\frac{ 1}{2},\frac{ 5}{6})$}&\pbox{20cm}{$\mathcal M_3(3,5)$,\\$\mathcal N_4(1,1,5)$}&Yes$^\ast$\\ \hline

\pbox{20cm}{${\scriptstyle \alpha=}(0,\frac{1}{14},\frac{3}{14},\frac{5}{14},\frac{9}{14},\frac{11}{14},\frac{13}{14})$ \\ ${\scriptstyle \beta=}(\frac{1}{7},\frac{ 2}{7},\frac{ 3}{7},\frac{ 1}{2},\frac{ 4}{7},\frac{ 5}{7},\frac{ 6}{7})$}&\pbox{20cm}{${\scriptstyle \alpha=}(\frac{1}{7},\frac{ 2}{7},\frac{ 3}{7},\frac{ 1}{2},\frac{ 4}{7},\frac{ 5}{7},\frac{ 6}{7})$ \\ ${\scriptstyle \beta=}(0,\frac{1}{14},\frac{3}{14},\frac{5}{14},\frac{9}{14},\frac{11}{14},\frac{13}{14})$}&$\mathcal N_1(1,1,7)$&Yes\\
\end{longtable}}
\end{center}
\let\thefootnote\relax\footnote{$^7$ H.~Park \cite{park} has recently shown that all the $R_f$'s in Table~3 are thin except for the second entry, which is arithmetic.  Moreover, in that case $H_r$ is still contained in a thin reflection subgroup of $R_f$.  In particular, all the $H_r$'s in Table~3 are thin.}
\subsection{Some numerics for hyperbolic hypergeometric monodromy groups}\label{numeric1}

While our certificate for being thin succeeds in many cases as demonstrated in Sections~\ref{mindistgraph} and \ref{numeric}, there are families such as $\mathcal N_4(j,k,n)$ for which we found almost no paths between the basis vectors in the minimal distance graph.  If indeed the corresponding components of $X_f$ are singletons for members of this family, then clearly a variation of the minimal distance graph is needed if this approach is to work.  A crude, and it appears mostly reliable, test for the size of $H=H(\alpha,\beta)$ in $O(L)$ is to simply count the number of elements in $H$ in a large ball.  That is, let $T$ be a large number and let
\begin{equation}\label{inball}
N_H(T):=\{\gamma\in H\; |\; \mbox{trace}(\gamma^t\gamma)\leq T^2\}.
\end{equation}
\noindent If $H$ is arithmetic then it is known (\cite{LP}) that
\begin{equation}\label{asymptotic}
N_H(T)\sim\frac{c_n\, T^{n-2}}{\mbox{Vol}(H\backslash \textrm{O}_f(\mathbb R))}, \mbox{ as } T\rightarrow\infty.
\end{equation}
\noindent Thus to probe the size of $H$ we generate elements in $H$ using our defining generators and try to determine $N_H(T)$ for $T$ quite large.

The difficulty is that we don't know which elements of $\textrm{O}(L)$ are in $H$ and the above procedure (by going to large generations of elements in $H$ gotten from the generating set) only gives a lower bound for $N_H(T)$.  In any case we can compute in this way a lower bound for $\log(N_H(T))/\log T$ for $T$ large.  If this is essentially $n-2$ then this suggests that $H$ is arithmetic, while if this number is less than $n-2$ it suggests that $H$ is thin (if $H$ is geometrically finite then this quantity should be an approximation to the Hausdorff dimension of the limit set of $H$ (see for example \cite{LP}).  We have run this crude test for many of our hyperbolic hypergeometric monodromies and it gives the correct answer in the cases where we have a rigorous treatment.  For example, in Figures~\ref{arith1} and \ref{arith2} we give a plot of $\log(N_H(T))$ versus $\log T$ for the arithmetic examples $\alpha=(\frac{1}{3},\frac{1}{2},\frac{2}{3}), \beta=(0,\frac{1}{4},\frac{3}{4})$ and $\alpha=(\frac{1}{3},\frac{1}{2},\frac{2}{3}), \beta=(0,\frac{1}{6},\frac{5}{6})$.  The similar plots for four examples with $n=5$, and for which our certificate failed are given in Figures~\ref{num1}, \ref{num2}, \ref{num3}, and \ref{num4} below.  Based on these, it appears that these $H$'s are all thin.  Further such experimentation is consistent with Conjecture~\ref{allthinconj} in the Introduction and even that for $n\geq 5$, every hyperbolic hypergeometric monodromy group is thin.

\begin{figure}
\centering
\includegraphics[height=55mm]{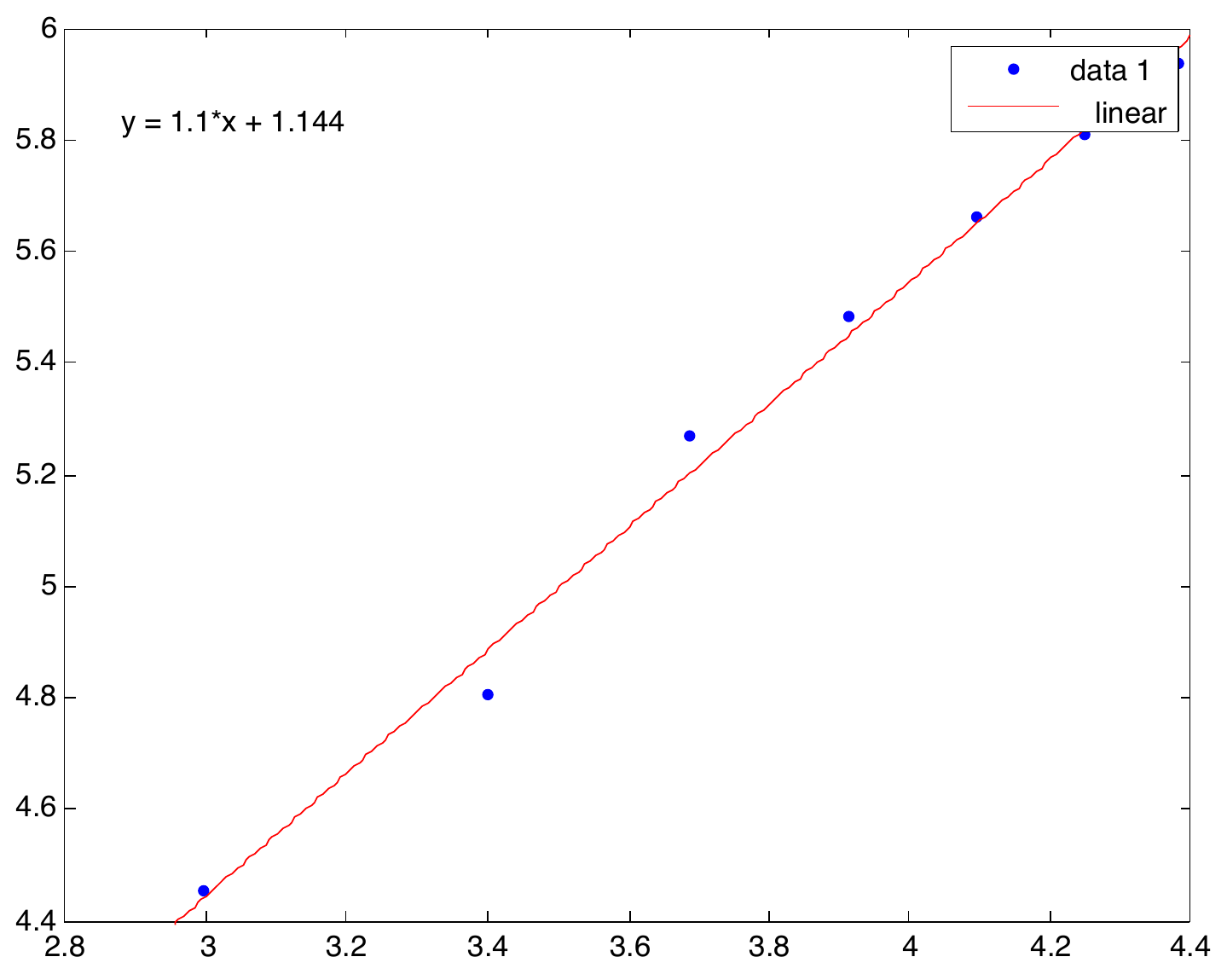}
\caption{$\log(N_H(T))$ versus $\log T$ graph for $H(\alpha,\beta)$ where $\alpha=(\frac{1}{3},\frac{1}{2},\frac{2}{3})$, $\beta=(0,\frac{1}{4},\frac{3}{4})$.}\label{arith1}
\end{figure}

\begin{figure}
\centering
\includegraphics[height=55mm]{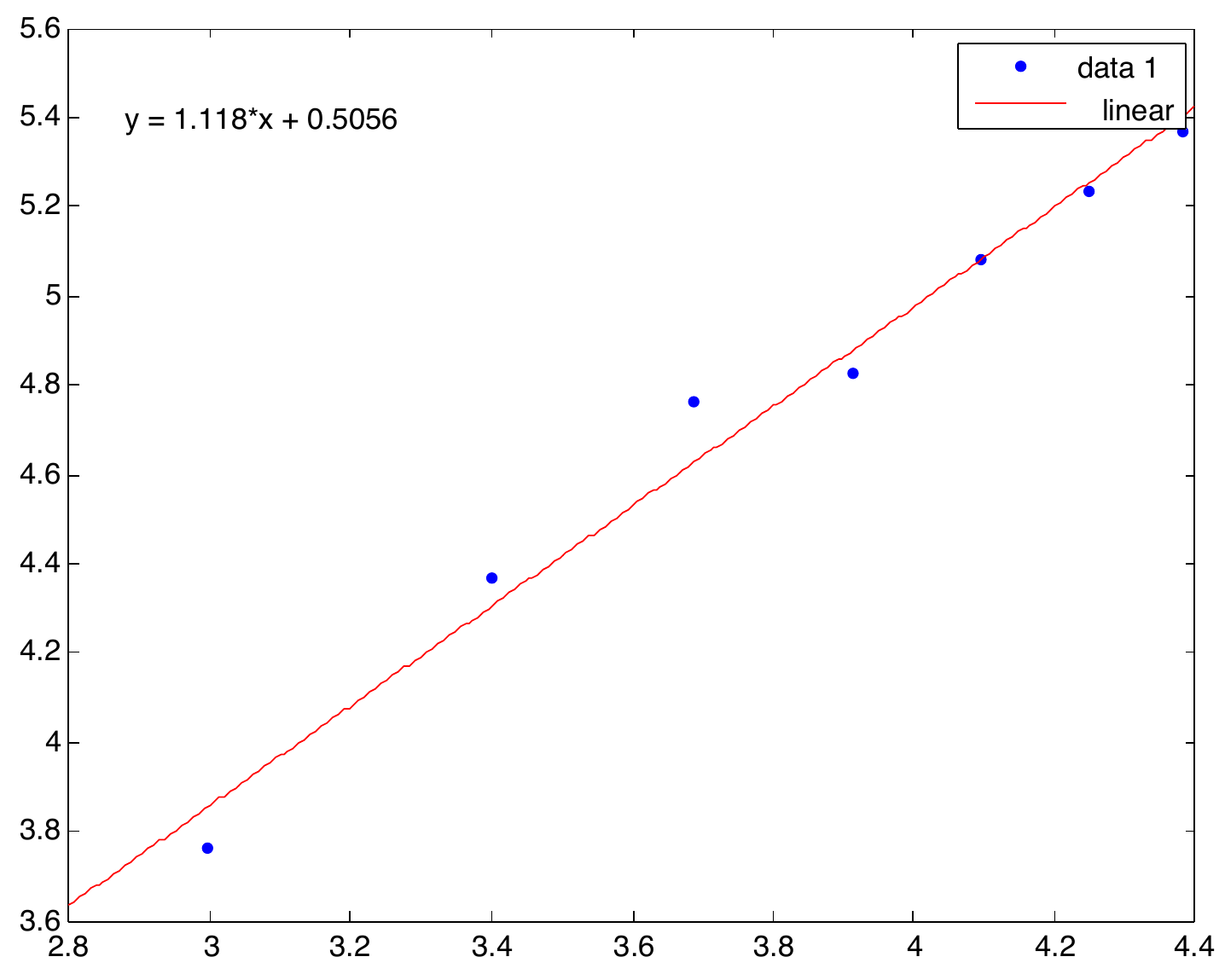}
\caption{$\log(N_H(T))$ versus $\log T$ graph for $H(\alpha,\beta)$ where $\alpha=(\frac{1}{3},\frac{1}{2},\frac{2}{3})$, $\beta=(0,\frac{1}{6},\frac{5}{6})$.}\label{arith2}
\end{figure}

\begin{figure}
\centering
\includegraphics[height=55mm]{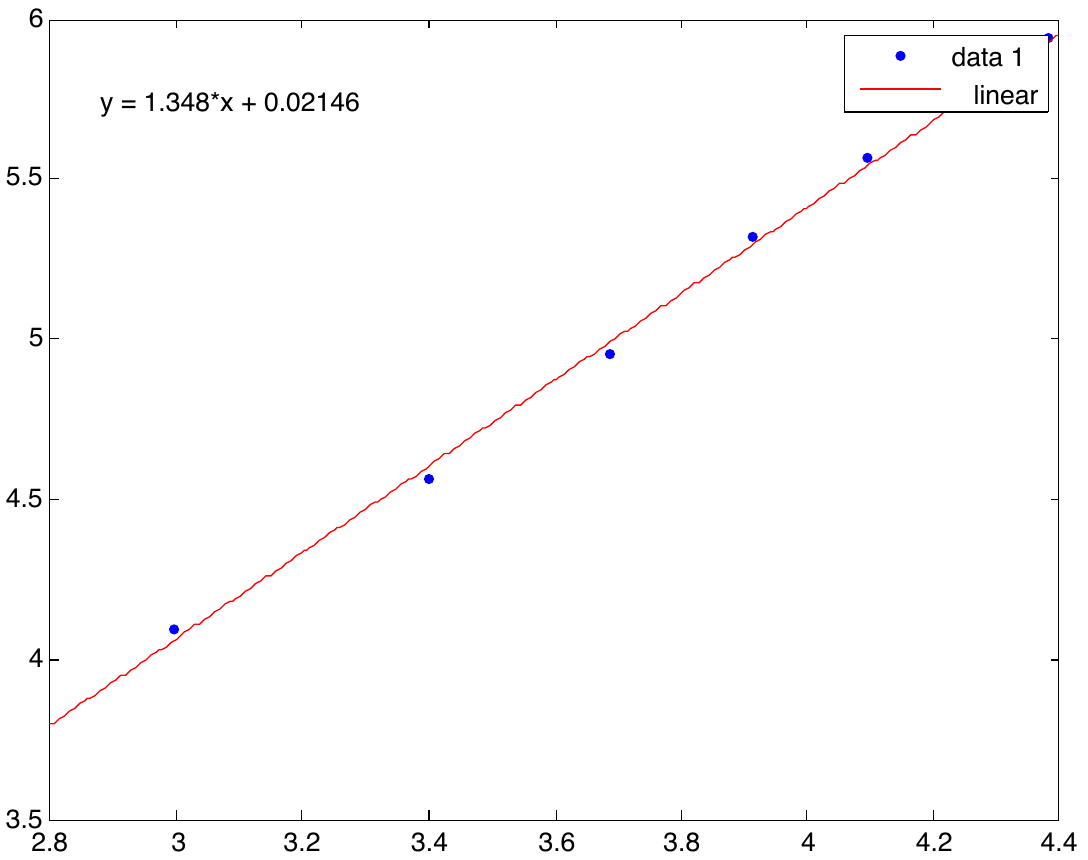}
\caption{$\log(N_H(T))$ versus $\log T$ graph for $H(\alpha,\beta)$ where $\alpha=(\frac{1}{5},\frac{2}{5},\frac{1}{2},\frac{3}{5},\frac{4}{5})$, $\beta=(0,0,0,\frac{1}{3},\frac{2}{3})$.}\label{num1}
\end{figure}
\begin{figure}
\centering
\includegraphics[height=55mm]{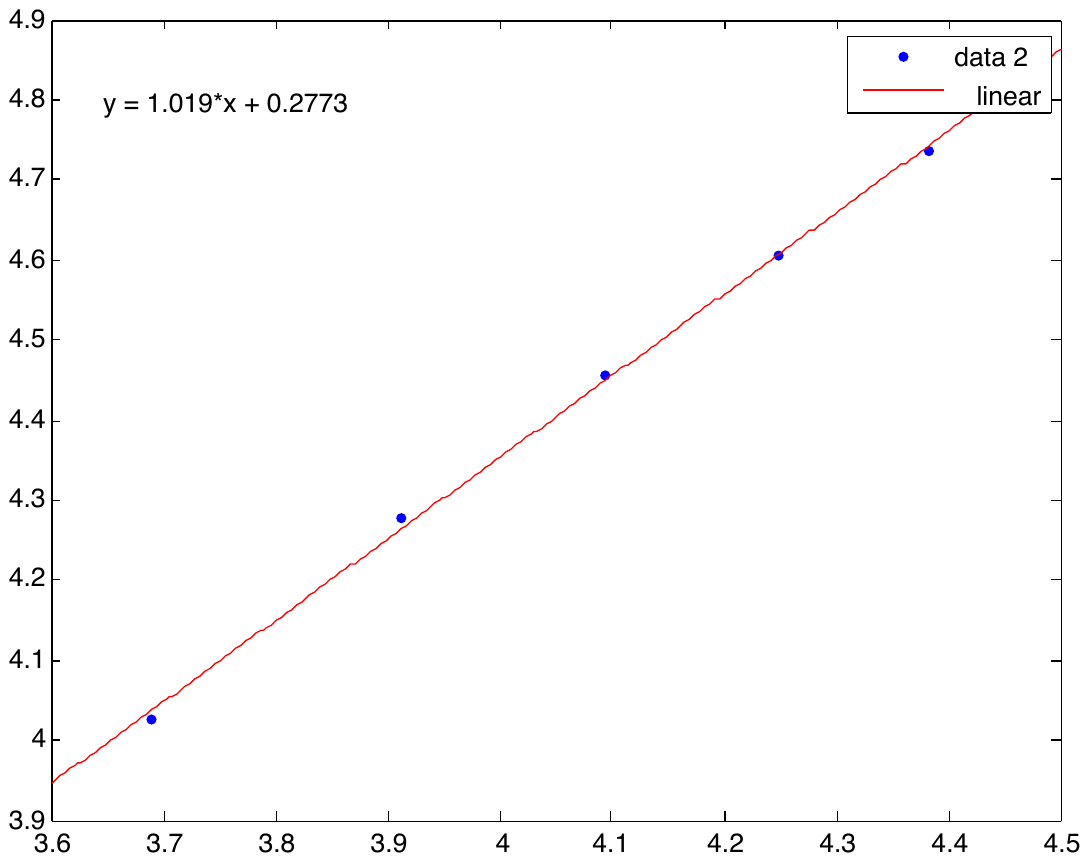}
\caption{$\log(N_H(T))$ versus $\log T$ graph for $H(\alpha,\beta)$ where $\alpha=(\frac{1}{8},\frac{3}{8},\frac{1}{2},\frac{5}{8},\frac{7}{8})$, $\beta=(0,0,0,\frac{1}{6},\frac{5}{6})$.}\label{num2}
\end{figure}
\begin{figure}
\centering
\includegraphics[height=55mm]{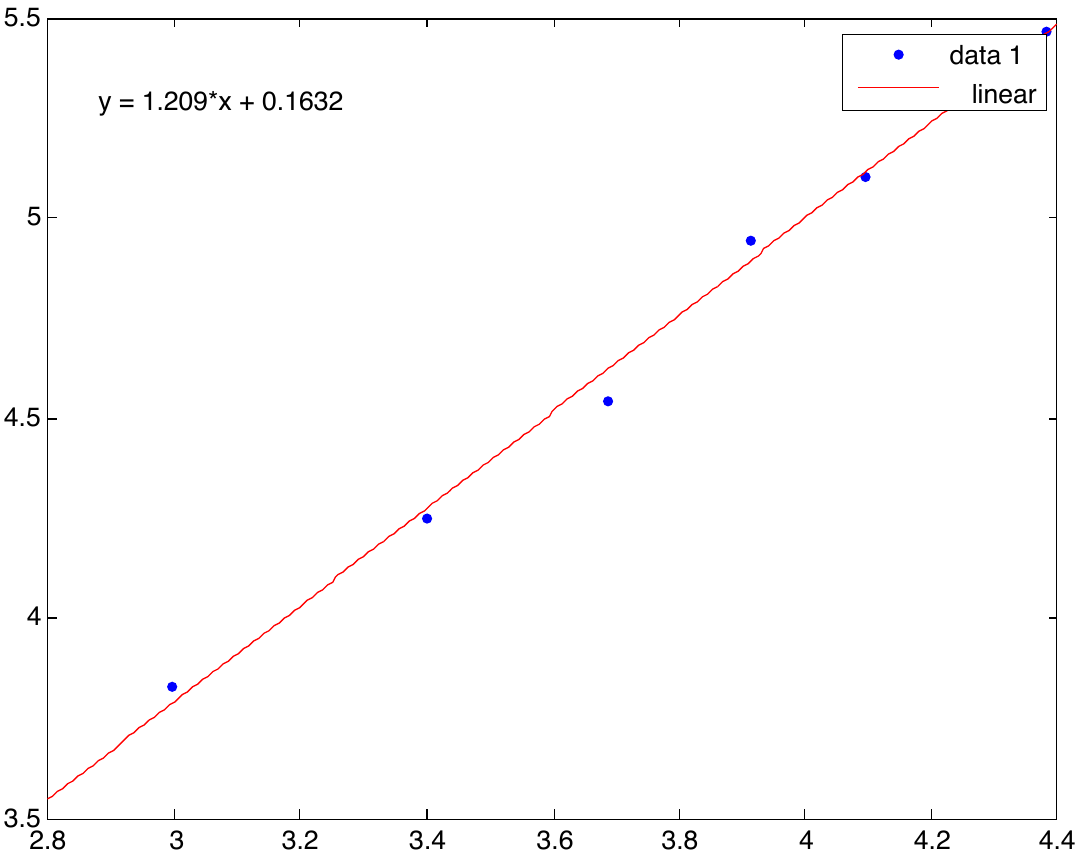}
\caption{$\log(N_H(T))$ versus $\log T$ graph for $H(\alpha,\beta)$ where $\alpha=(\frac{1}{6},\frac{1}{3},\frac{1}{2},\frac{2}{3},\frac{5}{6})$, $\beta=(0,0,0,\frac{1}{4},\frac{3}{4})$.}\label{num3}
\end{figure}
\begin{figure}
\centering
\includegraphics[height=55mm]{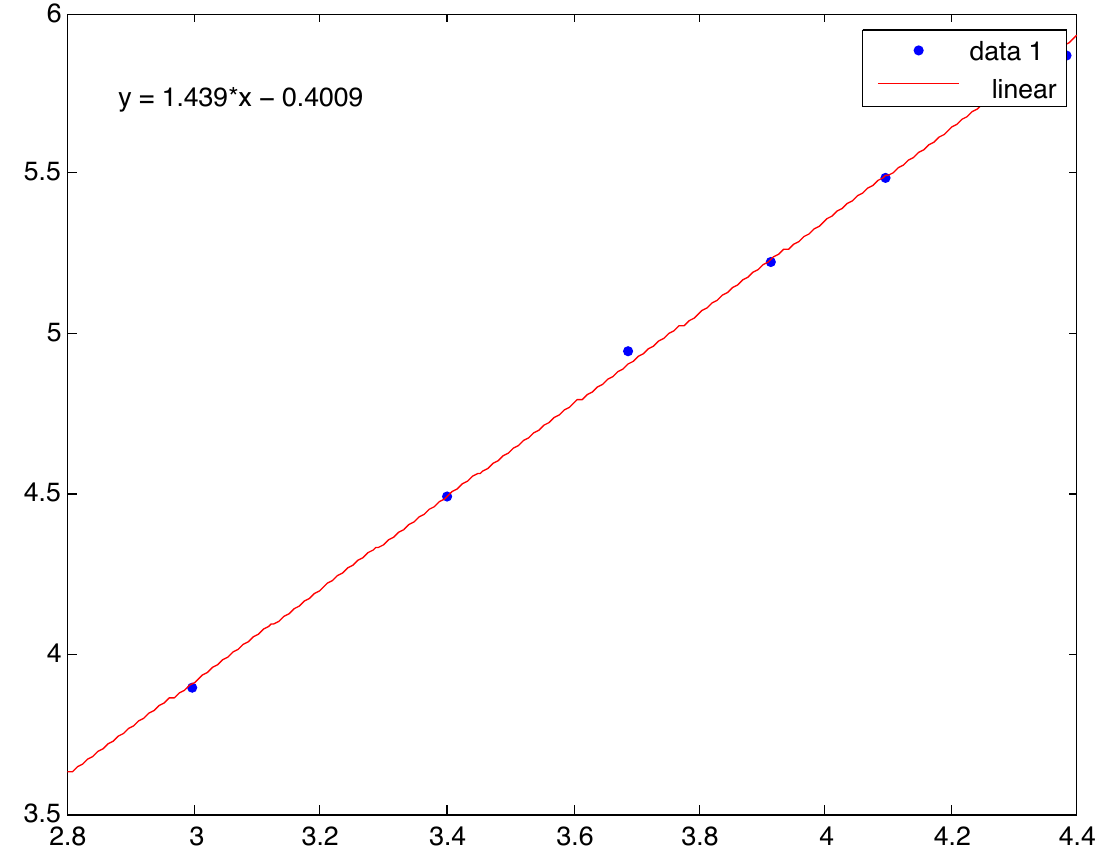}
\caption{$\log(N_H(T))$ versus $\log T$ graph for $H(\alpha,\beta)$ where $\alpha=(\frac{1}{6},\frac{1}{2},\frac{1}{2},\frac{1}{2},\frac{5}{6})$, $\beta=(0,0,0,\frac{1}{3},\frac{2}{3})$.}\label{num4}
\end{figure}

\vspace{3in}

\section*{Appendix}

In this appendix, we consider the six $3$-dimensional hyperbolic hypergeometric monodromy groups and show that they are all arithmetic.  Our method is to determine the preimage of the groups in the spin double cover $\textrm{SL}_2(\mathbb R)$ of $\textrm{SO}(2,1)$ and work in the setting of $\textrm{SL}_2$.  Once we pull back to $\textrm{SL}_2$, we rely on one of the following strategies.  If the quadratic form $f$ that is fixed by hyperbolic hypergeometric monodromy group $H$ is isotropic over $\mathbb Q$ then one can find $M\in \textrm{GL}_3(\mathbb Q)$ such that $M^tfM$ is some scalar multiple of $Q_2$ below.  Moreover, under such a change of variable, the preimage of $H\cap \textrm{SO}_f$ in the spin double cover of $\textrm{SO}_f$ can be made to sit inside $\textrm{SL}_2(\mathbb Z)$, and in this setting showing that the group is finite index is usually straightforward.  Out of the six $3$-dimensional hyperbolic hypergeometric monodromy groups, four fix an isotropic form $f$, and in those cases we show that the preimage of the group in the spin double cover of $\textrm{SO}_f$ contains a principle congruence subgroup\footnote[8]{$^8$In doing this, we use Sage to determine generators for $\Gamma(N)$ for various $N$.} of $\textrm{SL}_2(\mathbb Z).^8$

If the form $f$ fixed by $H$ is anisotropic, we again pull back to the spin double cover $\Gamma$ of $\textrm{SO}_f$ in $\textrm{SL}_2(\mathbb R)$ and, after generating several elements of $\Gamma$, construct a region in the upper half plane which contains a fundamental domain for $\Gamma$ acting on $\mathbb H$ (this method works for the isotropic case as well).  Specifically, a Dirichlet fundamental domain for $\Gamma$ is given by
$$\mathcal F_{p_0}=\bigcap_{\gamma\in \Gamma}\{z\in\mathbb H\; |\; d(z,p_0)\leq d(z,\gamma p_0)\}=\bigcap_{\gamma\in\Gamma} H(\gamma,p_0)$$
where $p_0$ is any base point and $H(\gamma,p_0)$ is a closed half space.  One can prove that $\mathcal F_{p_0}$ is compact (and hence $\Gamma$ is arithmetic) if the intersection $\cap_{\gamma} H(\gamma,p_0)$ is compact even when taken over a \emph{finite} number of elements $\gamma\in \Gamma$.  Note that this method uses the same idea as described in Section~\ref{numeric1}: generating elements of $\Gamma$ to say something about thinness, but in this case it can actually prove whether the group is thin or not.

Throughout this section, we let
$$Q_1=\left[\begin{array}{ccc}
1&0&0\\
0&1&0\\
0&0&\minus 1
\end{array}\right]\quad\mbox{ and }\quad
Q_2=\left[\begin{array}{ccc}
0&0&\minus\frac{1}{2}\\
0&1&0\\
\minus\frac{1}{2}&0&0
\end{array}\right].$$

For these two forms we have the following spin homomorphisms from $\textrm{SL}_2(\mathbb R)$ into the $\textrm{SO}_{Q_1}$ and $\textrm{SO}_{Q_2}$, respectively (see \cite{cas}).
\begin{equation}\label{spin1}
\begin{array}{lcl}
{\left(\begin{array}{ll}a&b\\ c&d\\ \end{array}\right)}&{\stackrel{\rho_1}{\longmapsto}}&{\left(\begin{array}{ccc}\frac{1}{2}(a^2+b^2-c^2+d^2)&ac-bd&\frac{1}{2}(a^2-b^2+c^2-d^2)\\ ab-cd&bc+ad&ab+cd\\ \frac{1}{2}(a^2+b^2-c^2-d^2)&ac+bd&\frac{1}{2}(a^2+b^2+c^2+d^2)\\ \end{array}\right)}\\
\end{array}
\end{equation}
\begin{equation}\label{spin2}
\begin{array}{lcl}
{\left(\begin{array}{ll}a&b\\ c&d\\ \end{array}\right)}&{\stackrel{\rho_2}{\longmapsto}}&{\left(\begin{array}{ccc}a^2&2ac&c^2\\ ab&ad+bc&cd\\ b^2&2bd&d^2\\ \end{array}\right)}\\
\end{array}
\end{equation}

\vspace{0.1in}

\noindent{\bf Example 1}: $\alpha=(\frac{1}{2},\frac{1}{2},\frac{1}{2}), \beta=(0,0,0)$.

In this case $H(\alpha,\beta)\subset \textrm{O}_f(\mathbb Z)$ with generators $A,B$ where

$$f=\left[\begin{array}{ccc}
1&0&\minus 3\\
0&1&0\\
\minus 3&0&1
\end{array}\right], A=\left[\begin{array}{ccc}
0&0&\minus 1\\
1&0&\minus 3\\
0&1&\minus 3
\end{array}\right], B=\left[\begin{array}{ccc}
0&0&1\\
1&0&\minus 3\\
0&1&3
\end{array}\right].$$

Note that $2(M^tfM)=Q_2$ where
$$M=\left[\begin{array}{ccc}
\minus 1/8&1/4&\minus 1/4\\
\minus 1/4&0&1/2\\
\minus 1/8&\minus 1/4&\minus 1/4
\end{array}\right]$$
and conjugation by $M$ gives an isomorphism of the subgroup $\langle A^2,B\rangle\subset H(\alpha,\beta)\cap \textrm{SO}_f(\mathbb Z)$ with the subgroup $\langle A',B'\rangle\subset \textrm{SO}_{Q_2}(\mathbb Z)$ where
$$A'=\left[\begin{array}{ccc}
1&8&16\\
0&1&4\\
0&0&1
\end{array}\right], B'=\left[\begin{array}{ccc}
1&0&0\\
1&1&0\\
1&2&1
\end{array}\right].$$

The preimage of this group in $\textrm{SL}_2(\mathbb Z)$ under the spin homomorphism in (\ref{spin2}) is the group $\langle \pm X,\pm Y\rangle$ where
$$X=\left[\begin{array}{cc}
1&0\\
4&1
\end{array}\right], Y=\left[\begin{array}{cc}
1&1\\
0&1
\end{array}\right].$$
This group contains the generators
$$\left[\begin{array}{cc}
1&4\\
0&1
\end{array}\right], \left[\begin{array}{cc}
\minus 15&4\\
\minus 4&1
\end{array}\right], \left[\begin{array}{cc}
5&\minus 4\\
4&\minus 3
\end{array}\right], \left[\begin{array}{cc}
9&\minus 16\\
4&\minus 7
\end{array}\right], \left[\begin{array}{cc}
13&\minus 36\\
4&\minus 11
\end{array}\right]$$
of $\Gamma(4)$. Hence $H(\alpha,\beta)$ is itself arithmetic.  In addition, our second strategy of finding a region which contains a fundamental domain for $H$ yields the compact region shown in Figure~\ref{fundy1}.

\vspace{0.1in}

\noindent{\bf Example 2}: $\alpha=(\frac{1}{3},\frac{1}{2},\frac{2}{3}), \beta=(0,0,0)$.

In this case $H(\alpha,\beta)\subset \textrm{O}_f(\mathbb Z)$ with generators $A,B$ where

$$f=\left[\begin{array}{ccc}
7&1&\minus 17\\
1&7&1\\
\minus 17&1&7
\end{array}\right], A=\left[\begin{array}{ccc}
0&0&\minus 1\\
1&0&\minus 2\\
0&1&\minus 2
\end{array}\right], B=\left[\begin{array}{ccc}
0&0&1\\
1&0&\minus 3\\
0&1&3
\end{array}\right].$$
Note that $(M^tfM)=3\cdot Q_2$ where
$$M=\left[\begin{array}{ccc}
\minus 1/4&0&1/12\\
\minus 1/2&1/2&1/12\\
1/4&\minus 1/2&1/3
\end{array}\right]$$
and conjugation by $M$ gives an isomorphism of the subgroup $\langle A^2,B\rangle\subset H(\alpha,\beta)\cap \textrm{SO}_f(\mathbb Z)$ with the subgroup $\langle A',B'\rangle\subset \textrm{SO}_{Q_2}(\mathbb Z)$ where
$$A'=\left[\begin{array}{ccc}
1&\minus 2&1\\
3&\minus 5&2\\
9&\minus 12&4
\end{array}\right], B'=\left[\begin{array}{ccc}
1&\minus 2&1\\
0&1&\minus 1\\
0&0&1
\end{array}\right].$$

The preimage of this group in $\textrm{SL}_2(\mathbb Z)$ under the spin homomorphism in (\ref{spin2}) is the group $\langle \pm X,\pm Y\rangle$ where
$$X=\left[\begin{array}{cc}
\minus 1&\minus 3\\
1&2
\end{array}\right], Y=\left[\begin{array}{cc}
1&0\\
\minus 1&1
\end{array}\right].$$
This group contains the generators
$$\left[\begin{array}{cc}
1&3\\
0&1
\end{array}\right], \left[\begin{array}{cc}
\minus 8&3\\
\minus 3&1
\end{array}\right], \left[\begin{array}{cc}
4&\minus 3\\
3&\minus 2
\end{array}\right]$$
of $\Gamma(3)$. Hence $H(\alpha,\beta)$ is itself arithmetic.  In addition, our second strategy of finding a region which contains a fundamental domain for $H$ yields the compact region shown in Figure~\ref{fundy2}.

\vspace{0.1in}

\noindent{\bf Example 3}: $\alpha=(\frac{1}{4},\frac{1}{2},\frac{3}{4}), \beta=(0,0,0)$.

In this case $H(\alpha,\beta)\subset \textrm{O}_f(\mathbb Z)$ with generators $A,B$ where

$$f=\left[\begin{array}{ccc}
3&1&\minus 5\\
1&3&1\\
\minus 5&1&3
\end{array}\right], A=\left[\begin{array}{ccc}
0&0&\minus 1\\
1&0&\minus 1\\
0&1&\minus 1
\end{array}\right], B=\left[\begin{array}{ccc}
0&0&1\\
1&0&\minus 3\\
0&1&3
\end{array}\right].$$
Note that $M^tfM=Q_1$ where
$$M=\left[\begin{array}{ccc}
1/4&1/4&1/2\\
0&1/2&0\\
\minus 1/4&1/4&1/2
\end{array}\right]$$
and conjugation by $M$ gives an isomorphism of the subgroup $\langle A^2,B\rangle\subset H(\alpha,\beta)\cap \textrm{SO}_f(\mathbb Z)$ with the subgroup $\langle A',B'\rangle\subset \textrm{SO}_{Q_1}(\mathbb Z)$ where
$$A'=\left[\begin{array}{ccc}
\minus 1&0&0\\
0&\minus 1&0\\
0&0&1
\end{array}\right], B'=\left[\begin{array}{ccc}
1&\minus 2&\minus 2\\
2&\minus 1&\minus 2\\
\minus 2&2&3
\end{array}\right].$$

The preimage of this group in $\textrm{SL}_2(\mathbb Z)$ under the spin homomorphism in (\ref{spin1}) is the group $\langle \pm X,\pm Y\rangle$ where
$$X=\left[\begin{array}{cc}
0&1\\
\minus 1&0
\end{array}\right], Y=\left[\begin{array}{cc}
0&1\\
\minus 1&2
\end{array}\right]$$
and $\pm XY, \pm YX$ generate the principal congruence subgroup $\Gamma(2)$. Hence $H(\alpha,\beta)$ is itself arithmetic.  In addition, our second strategy of finding a region which contains a fundamental domain for $H$ yields the compact region shown in Figure~\ref{fundy3}.

\vspace{0.1in}

\noindent{\bf Example 4}: $\alpha=(\frac{1}{6},\frac{1}{2},\frac{5}{6}), \beta=(0,0,0)$.

In this case $H(\alpha,\beta)\subset \textrm{O}_f(\mathbb Z)$ with generators $A,B$ where

$$f=\left[\begin{array}{ccc}
5&3&\minus 3\\
3&5&3\\
\minus 3&3&5
\end{array}\right], A=\left[\begin{array}{ccc}
0&0&\minus 1\\
1&0&0\\
0&1&0
\end{array}\right], B=\left[\begin{array}{ccc}
0&0&1\\
1&0&\minus 3\\
0&1&3
\end{array}\right].$$
Note that $M^tfM=Q_1$ where
$$M=\left[\begin{array}{ccc}
1/4&\minus 1/2&3/4\\
1/4&1/2&\minus 3/4\\
0&0&1/2
\end{array}\right]$$
and conjugation by $M$ gives an isomorphism of the subgroup $\langle A^2,B\rangle\subset H(\alpha,\beta)\cap \textrm{SO}_f(\mathbb Z)$ with the subgroup $\langle A',B'\rangle\subset \textrm{SO}_{Q_1}(\mathbb Z[\frac{1}{2})$ where
$$A'=\left[\begin{array}{ccc}
\minus 1/2&\minus 1&1/2\\
1&\minus 1&1\\
1/2&\minus 1&3/2
\end{array}\right], B'=\left[\begin{array}{ccc}
1/2&\minus 1&\minus 1/2\\
1&1&1\\
1/2&1&3/2
\end{array}\right].$$

The preimage of this group in $\textrm{SL}_2(\mathbb Z)$ under the spin homomorphism in (\ref{spin1}) is the group $\langle \pm X,\pm Y\rangle$ where
$$X=\left[\begin{array}{cc}
1&1\\
\minus 1&0
\end{array}\right], Y=\left[\begin{array}{cc}
1&1\\
0&1
\end{array}\right]$$
and $YX^4, Y$ are a well-known generating set for all of $\textrm{SL}_2(\mathbb Z)$. Hence $H(\alpha,\beta)$ is itself arithmetic.  In addition, our second strategy of finding a region which contains a fundamental domain for $H$ yields the compact region shown in Figure~\ref{fundy4}.

\vspace{0.1in}

\noindent{\bf Example 5}: $\alpha=(\frac{1}{3},\frac{1}{2},\frac{2}{3}), \beta=(0,\frac{1}{4},\frac{3}{4})$.

In this case $H(\alpha,\beta)\subset \textrm{O}_f(\mathbb Z)$ with generators $A,B$ where

$$f=\left[\begin{array}{ccc}
5&\minus 1&\minus 7\\
\minus 1&5&\minus 1\\
\minus 7&\minus 1&5
\end{array}\right], A=\left[\begin{array}{ccc}
0&0&\minus 1\\
1&0&\minus 2\\
0&1&\minus 2
\end{array}\right], B=\left[\begin{array}{ccc}
0&0&1\\
1&0&\minus 1\\
0&1&1
\end{array}\right].$$
\noindent In this case $f$ is anisotropic over $\mathbb Q$, and hence we turn to our strategy of approximating the fundamental domain to prove that $H$ is arithmetic: we produce the compact region shown in Figure~\ref{fundy5} which contains the fundamental domain of $H$.

\vspace{0.1in}

\noindent{\bf Example 6}: $\alpha=(\frac{1}{3},\frac{1}{2},\frac{2}{3}), \beta=(0,\frac{1}{6},\frac{5}{6})$.

In this case $H(\alpha,\beta)\subset \textrm{O}_f(\mathbb Z)$ with generators $A,B$ where

$$f=\left[\begin{array}{ccc}
1&0&\minus 2\\
0&1&0\\
\minus 2&0&1
\end{array}\right], A=\left[\begin{array}{ccc}
0&0&\minus 1\\
1&0&\minus 2\\
0&1&\minus 2
\end{array}\right], B=\left[\begin{array}{ccc}
0&0&1\\
1&0&\minus 2\\
0&1&2
\end{array}\right].$$
\noindent In this case $f$ is anisotropic over $\mathbb Q$, and hence we turn to our strategy of approximating the fundamental domain to prove that $H$ is arithmetic: we produce the compact region shown in Figure~\ref{fundy6} which contains the fundamental domain of $H$.

\begin{figure}
\centering
\includegraphics[height=33mm]{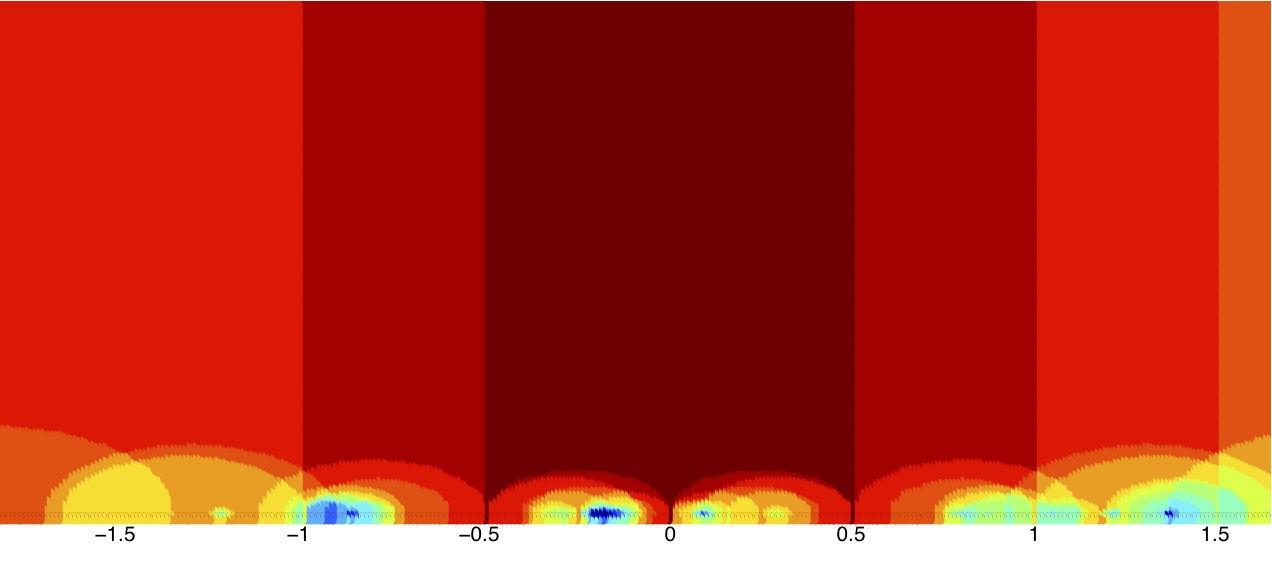}
\caption{Dark red region in upper half plane containing a fundamental domain of $H(\alpha,\beta)$ in Example 1.}\label{fundy1}
\end{figure}
\begin{figure}
\centering
\includegraphics[height=33mm]{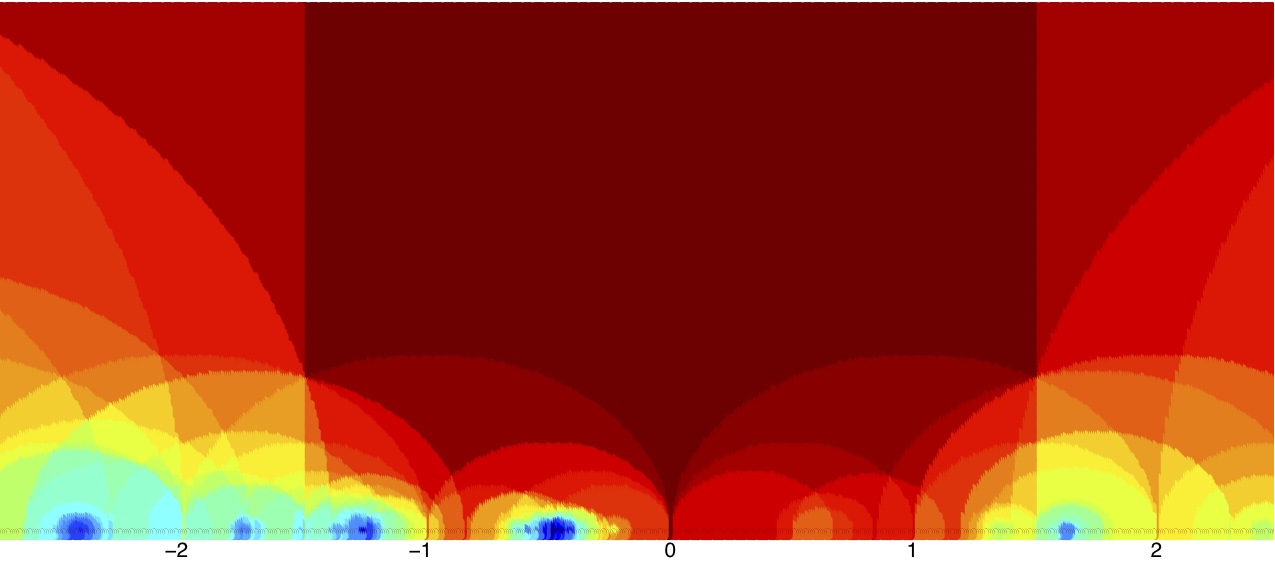}
\caption{Dark red region in upper half plane containing a fundamental domain of $H(\alpha,\beta)$ in Example 2.}\label{fundy2}
\end{figure}
\begin{figure}
\centering
\includegraphics[height=33mm]{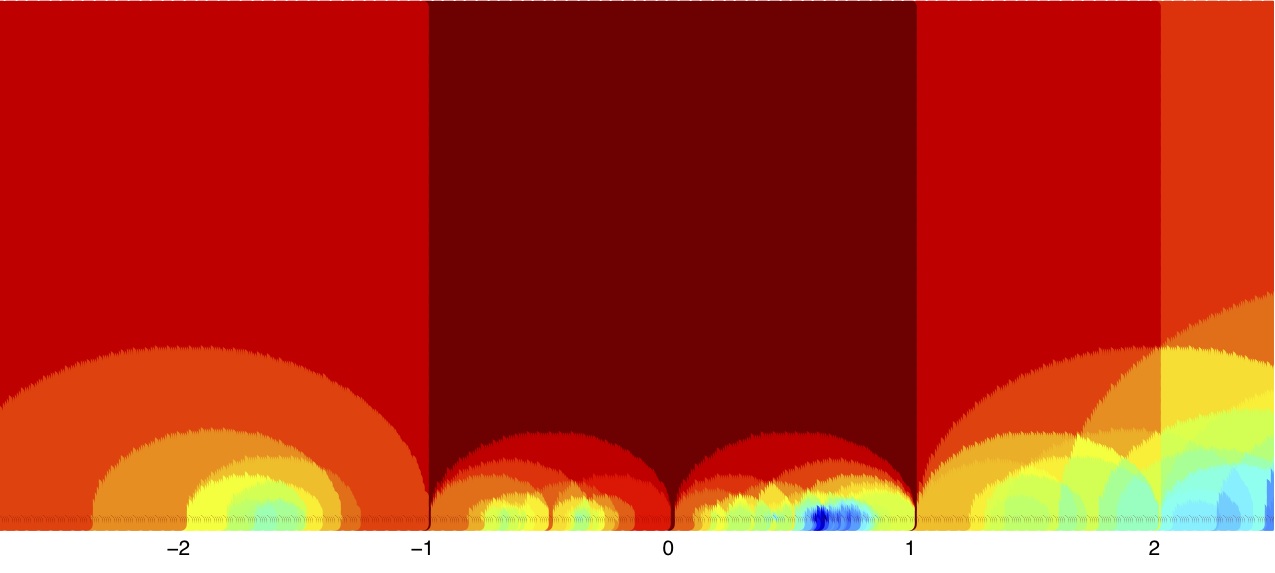}
\caption{Dark red region in upper half plane containing a fundamental domain of $H(\alpha,\beta)$ in Example 3.}\label{fundy3}
\end{figure}
\begin{figure}
\centering
\includegraphics[height=33mm]{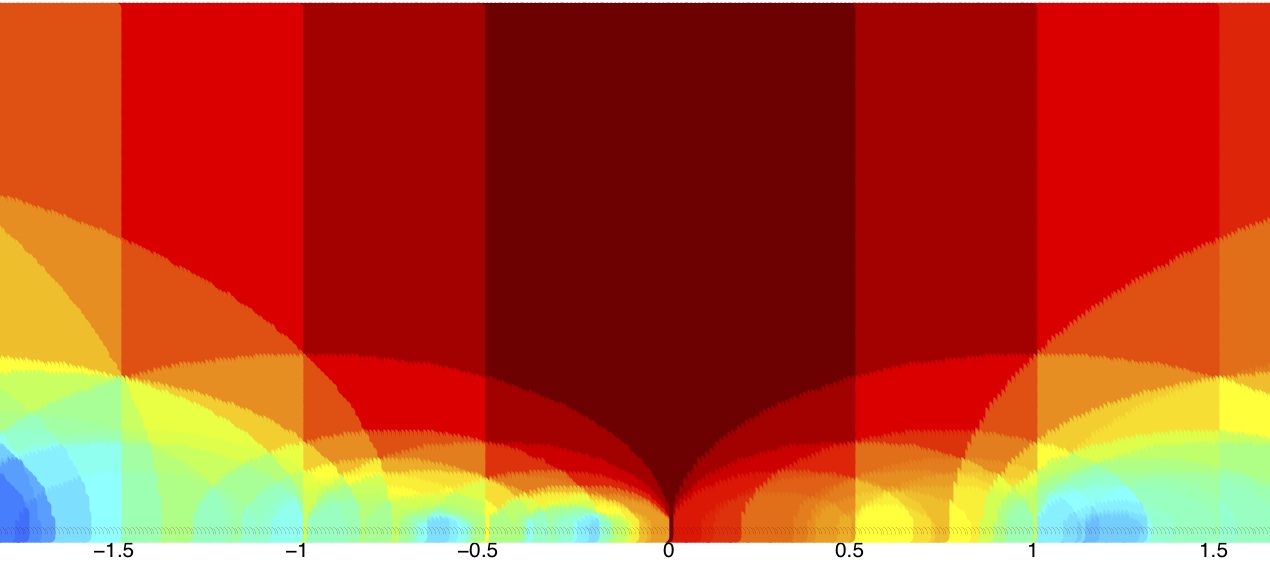}
\caption{Dark red region in upper half plane containing a fundamental domain of $H(\alpha,\beta)$ in Example 4.}\label{fundy4}
\end{figure}
\begin{figure}
\centering
\includegraphics[height=33mm]{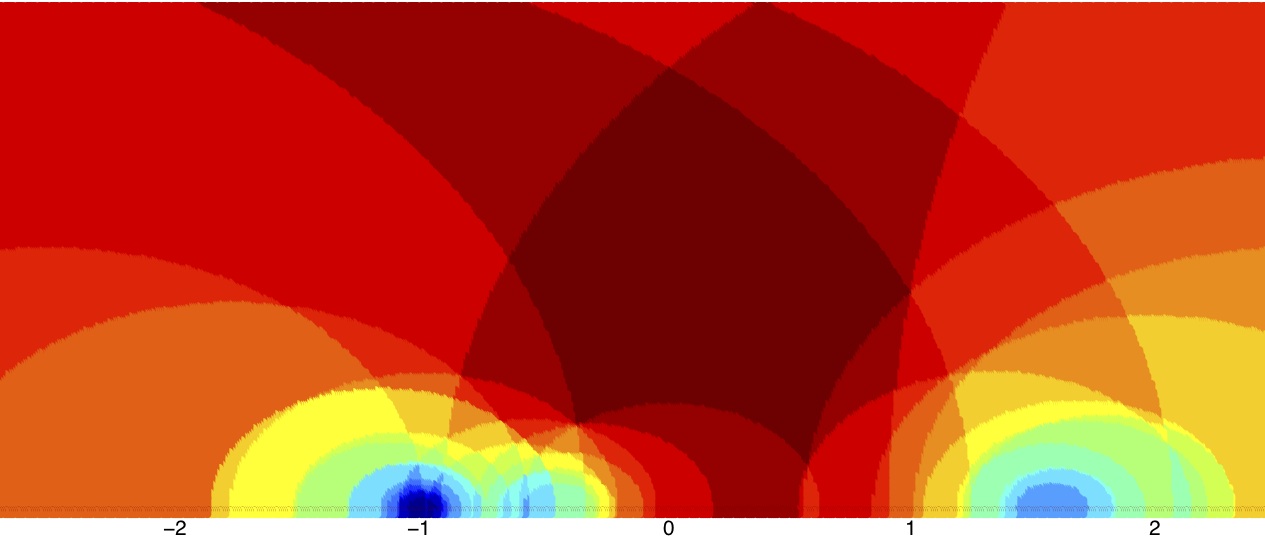}
\caption{Dark red region in upper half plane containing a fundamental domain of $H(\alpha,\beta)$ in Example 5.}\label{fundy5}
\end{figure}
\begin{figure}
\centering
\includegraphics[height=33mm]{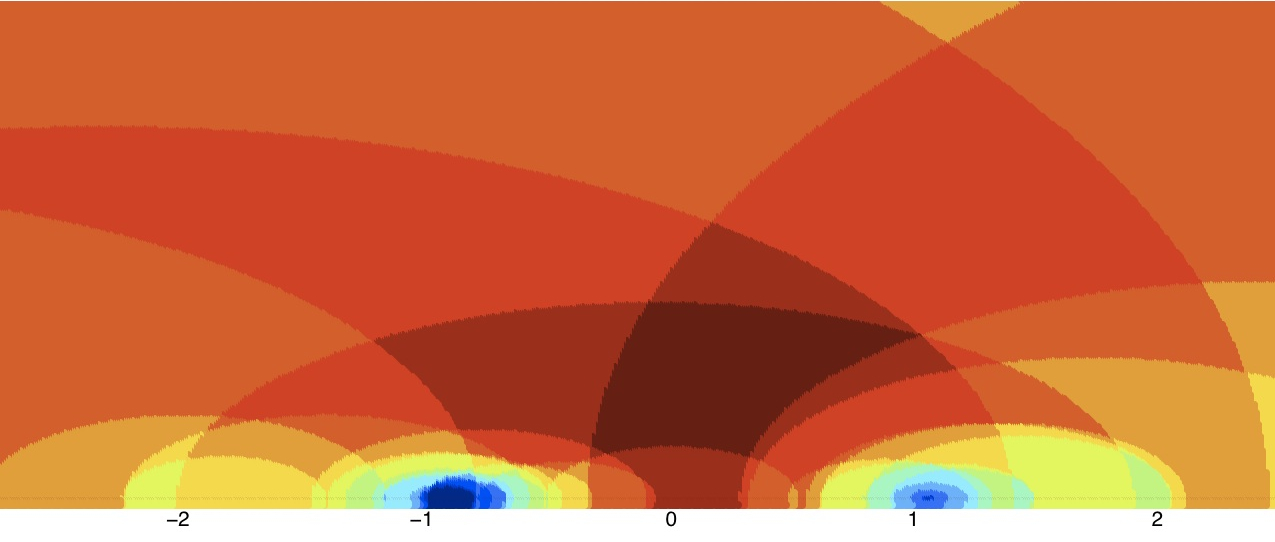}
\caption{Dark red region in upper half plane containing a fundamental domain of $H(\alpha,\beta)$ in Example 6.}\label{fundy6}
\end{figure}

\vspace{3in}

\noindent{\bf Acknowledgements}: We thank I.~Capdeboscq, N.~Katz, C.~McMullen, V.~Nikulin, I.~Rivin, T.~Venkataramana, and A.~Wienhard for helpful discussions related to this paper.

\end{document}